\RequirePackage{etoolbox}
\csdef{input@path}{{style/}{graphics/}}
\documentclass[MSNbibl,number,citesort,seceqn,dvips]{arxbj}
\usepackage{upgreek}
\usepackage{graphicx}

\aid{0}
\volume{21}
\issue{1}
\pubyear{2015}
\firstpage{465}
\lastpage{488}
\doi{10.3150/13-BEJ575} %kopijuoti is PTS

\makeatletter

\newremark{Remark}{Remark}
\newtheorem{theorem}{Theorem}[section]
\newtheorem{lem}[theorem]{Lemma}
\newproclaim{defn}{Definition}[section]

\newcommand{\eqref}[1]{(\ref{#1})}
\newcommand{\AR}{\operatorname{AR}}
\renewcommand{\Pr}{\operatorname{Pr}}
\renewcommand{\epsilon}{\varepsilon}

\def\afrac#1#2{#1/(#2)}

\def\sklfrac#1#2{(#1/#2)}

\makeatother

\begin{document}
\begin{frontmatter}

\title{Testing the regularity of a smooth signal}
\runtitle{Testing the regularity of a smooth signal}

\begin{aug}
%%%% inicialai - be tarpu
\author{\inits{A.}\fnms{Alexandra} \snm{Carpentier}\corref{}\ead[label=e1]{a.carpentier@statslab.cam.ac.uk}}% \and
%%\runauthor{} %% auto
\address{Statistical Laboratory, Centre for Mathematical Sciences,
Wilberforce Road, CB3 0WB Cambridge, United Kingdom. \printead{e1}}
\end{aug}

% HISTORY:
\received{\smonth{4} \syear{2013}}
\revised{\smonth{8} \syear{2013}}

% ABSTRACT
%
\begin{abstract}
We develop a test to determine whether a function lying in a fixed
$L_2$-Sobolev-type ball of
smoothness $t$, and generating a noisy signal, is in fact of a given
smoothness $s \geq t$
or not. While it is impossible to construct a uniformly consistent test
for this problem on
\textit{every} function of smoothness $t$, it becomes possible if we
remove a sufficiently large region
of the set of functions of smoothness $t$. The functions that we remove
are functions of smoothness
strictly smaller than $s$, but that are very close to $s$-smooth
functions. A lower bound on the size
of this region has been proved to be of order $n^{-t/(2t+1/2)}$, and in
this paper, we provide a test
that is consistent after the removal of a region of such a size. Even
though the null hypothesis is
composite, the size of the region we remove does not depend on the
complexity of the null
hypothesis.
\end{abstract}

% KEYWORDS
% visi is mazosios raides ir pagal abecele
%

\end{frontmatter}

%s1 #&#
\section{Introduction}

We consider in this paper a composite testing problem in the
non-parametric Gaussian regression setting. % We want to decide whether
%the function generating the data is of some given smoothness $s$, or
%whether it is rougher (of some smoothness $t<s$), by constructing a
%smoothness test.% The difficulty of this problem comes mainly from the
%fact that the space of smooth functions, for instance Sobolev-type
%spaces, are of infinite dimension. It implies that traditional
%Neymann-Pearson theory does not apply for a such problem.
%We implement the test on two
Assuming that the unknown regression function $f$ lies in a given
smoothness class (indexed by $t$), we want to decide whether $f$ is in
fact in a much more regular class (indexed by $s \geq t$), by
constructing a suitable test. More precisely, we consider the setting
of testing between two fixed $L_2$ Sobolev-type classes, which we
define formally in Section~\ref{sec:setting} below.

%We will now present more formally the testing problem. %and also the
%previously achieved results i
%We consider in this paper the setting of testing between two fixed
%$L_2$ Sobolev-type classes (we define formally this object in Section

Let $\Sigma(t,B)$ be the $L_2$-Sobolev-type ball of functions in
$[0,1]$ of smoothness $t$ and radius $B$, and let $\Sigma(s,B)$ with
$s>t$ be a sub-model (i.e., $\Sigma(s,B) \subset\Sigma(t,B)$). We
assume that we have observations generated according to a Gaussian
non-parametric model with underlying function~$f$, at noise level $n$,
where $f\in\Sigma(s,B)$ or $f \in\Sigma(t,B) \setminus\Sigma(s,B)$.

For $G \subset L_2$, set $\|f - G\|_2 = \inf_{g \in G} \|f-g\|_2$. We
define for $\rho_n \geq0$ the sets
\begin{eqnarray*}
\tilde\Sigma(t, B,\rho_n) = \bigl\{f \in\Sigma(t,B)\dvtx  \bigl\|f -
\Sigma (s,B)\bigr\|_2 \geq\rho_n \bigr\}.
\end{eqnarray*}
Note that these sets are separated away from $\Sigma(s,B)$ whenever
$\rho_n >0$. They correspond to $\Sigma(t,B) \setminus\Sigma(s,B)$
where we have removed some critical functions, very close to functions
in $\Sigma(s,B)$.

We are interested in the composite testing problem:
%
%e1.1 #&#
\begin{equation}
\label{test} H_0\dvtx  f \in\Sigma(s,B)\quad  \mbox{vs.}\quad  H_1\dvtx  f \in\tilde
\Sigma(t, B,\rho_n).
\end{equation}
More precisely, we want to know the minimal order of magnitude of $\rho
_n$ that enables the construction of a uniformly consistent test $\Psi
_n$ between $H_0$ and $H_1$, that is, of a test such that there exists
$N$ that depends on $H_0, H_1$ and $\alpha$ only such that for any
$n\geq N$,
\begin{eqnarray*}
\sup_{f \in H_0} \mathbb E_{f} \Psi_n + \sup
_{f \in H_1} \mathbb E_{f} (1-\Psi_n ) \leq
\alpha.
\end{eqnarray*}

Two topics that are closely related to this question have been
thoroughly studied. The first one is non-parametric signal detection
where $H_0 = \{0\}$. The second is the creation of adaptive and honest
non-parametric confidence bands around functions.

Let us first recall the results obtained in signal detection where one
wishes to test%. The objective here is to test on whether the function
%generating the data is $0$ (no signal), or if this function is some
%smooth function that is non-null and that is not too close to $0$. In
%other words, one wishes to test
%
%e1.2 #&#
\begin{equation}
\label{test2} H_0\dvtx  f=0\quad  \mbox{vs.}\quad  H_1\dvtx  f \in \bigl\{f \in
\Sigma(t,B)\dvtx  \|f - 0\|_2 \geq \rho _n \bigr\}.
\end{equation}
As in any testing problem, in order to obtain a uniformly consistent
test, the model has to be restricted such that the elements in $H_0$
are not too close to the ones in $H_1$. This explains the presence of
the separation by $\rho_n$. Ingster \cite
{ingster1987minimax,ingster1993asymptotically}, Spokoiny \cite
{spokoiny1996adaptive} and
Ingster and Suslina \cite{ingster2002nonparametric}
prove that the minimal order of $\rho_n$ that
enables the existence of a consistent test in the above problem is
\begin{eqnarray*}
\rho_n \geq D n^{-\afrac{t}{2t+1/2}}.
\end{eqnarray*}
%
%Also, in paper \citep{spokoiny1996adaptive},
For $\rho_n$ of this order, the authors also build a consistent test
for the testing problem \eqref{test2}. They prove that the testing
problem is equivalent to testing whether the sum of the squares of the
means of independent (or close to independent) sub-Gaussian random
variables is null or not, and the usual $\chi_2$-test theory applies.
The size $\rho_n$ of the separation area is related to the minimax rate
of estimation of the $L_2$ norm of $f$ under the alternative
hypothesis. %, i.e. in $\Sigma(t,B)$.
%This problem is thus to our mind solved in a satisfying way.
A question that arises is how the results change when the null
hypothesis is a composite hypothesis, in our case an infinite
dimensional Sobolev-type ball.
%However, although it is related to the the testing problem that
%interests us described in Equation \eqref{test}, it is not equivalent,
%since in this case, the null hypothesis is an infinite dimensional
%Sobolev-type ball, whereas in the case of \citep{ingster1987minimax,
%ingster1993asymptotically,
%ingster2002nonparametric,spokoiny1996adaptive}, it is a single point ($

%what is in them. equation (1.1) is related to confidence bands only in
%the work Hoffmann and Nickl (2011) -- which you don't cite here -- and
%in Bull and Nickl (2013), and in some sense also in Hoffmann-Lepski
%and Juditsky & Lambert-Lacroix. These can be cited when you connect to
%the testing problem (1.1). That parameters have to be removed is
%proved in Low (1997), Cai and Low (2004), Hoffmann and Nickl (2011),
%Bull and Nickl (2013). The papers Cai and Low (2006), Baraud (2004),
%Robins van der Vaart (2006) and Bull and Nickl (2013) give
%constructive procedures without any removal in special cases. I don't
%know how you are going to discuss this literature but the way it is
%done now doesn't make sense, so please look at it carefully.}

%hoffman2011adaptive, bull2011adaptive}

%This is the reason why the results obtained for the testing problem
%described in Equation \eqref{test2} do not transfer trivially to the
%problem described in Equation \eqref{test}. Actually, t
The testing problem described in equation \eqref{test} is also closely
connected to the problem of the creation of confidence bands around
functions -- see, for instance, Hoffmann and Lepski \cite
{hoffman2002random},
Juditsky and Lambert-Lacroix \cite
{juditsky2003nonparametric}, Hoffmann and Nickl \cite
{hoffmann2011adaptive}, Bull and Nickl \cite{bull2011adaptive}
where this relation is made clear. %Since confidence bands have
%received a lot of attention recently, with some very nice results
%achieved, we are going to describe quickly this problem here, and how
%these results transfer to our problem.\\
Despite the fact that
%The puzzling fact for the creation of confidence bands, which is
%actually not so strange when one thinks in terms of testing problem, is
there exists a quite complete and satisfying theory for adaptive
non-parametric estimation % - a fundamental result is that without
%knowing the smoothness of a function, it is possible to estimate it as
%well as what an oracle knowing this smoothness would have performed,
%and that uniformly on \textit{every} function of some given Hilbert
%space -
-- see, for example, Lepski \cite
{lepski1992problems}, Donoho \textit{et al.} \cite{donoho1996density},
Barron \textit{et al.} \cite{barron1999risk}, Tsybakov \cite{tsybakov2003introduction} -- the
theory of adaptive
confidence sets has some fundamental limitations. Indeed, one has to
remove critical regions from the parameter space in order to construct
honest adaptive confidence sets, see Low
\cite{low1997nonparametric},
Cai and Low \cite{cai2004adaptation}, Hoffmann and Nickl \cite{hoffmann2011adaptive}, Bull and Nickl \cite{bull2011adaptive}. In the
paper Bull and Nickl \cite{bull2011adaptive}, the
problem of $L_2$-adaptive and honest
confidence sets is considered and in the course of the proofs, the
authors establish that in the testing problem \eqref{test}, $\rho_n$
can be taken of the order
\begin{eqnarray*}
\rho_n \geq D \max\bigl(n^{-\afrac{t}{2t+1/2}}, n^{-\afrac{s}{2s+1}}\bigr),
\end{eqnarray*}
for $D$ large enough depending on the level of the test and on $s,t$.
%%one can then build a consistent test for the testing problem
%described in Equation \eqref{test}.
On the other hand, they prove in the case of density estimation (we
provide a proof of this fact in our setting, see Theorem~\ref{th:lb}
below) that the lower bound for $\rho_n$ is
\begin{eqnarray*}
\rho_n \geq D' n^{-\afrac{t}{2t+1/2}},
\end{eqnarray*}
for some $D'$ positive; otherwise there exists no consistent test for
the problem \eqref{test}. In the case $s <2t$, the upper and lower
bound do not match (which in the context of confidence sets is
unimportant, see Baraud \cite
{baraud2004confidence}, Cai and Low \cite{cai2006adaptive},
Robins and Van Der Vaart \cite{robins2006adaptive},
Bull and Nickl \cite{bull2011adaptive} related results).
%As one can see, these two order of $\rho_n$ coincide whenever $2t \leq
%s$, and in that case the testing problem is thus solved in a minimax
%way, where again the order of the diameter of the separation area,
%i.e. $\rho_n$, is related to the minimax rate of estimation of the
%$L_2$ norm of $f$ under the alternative hypothesis, i.e. in $
%significantly smaller than the necessary $\rho_n$ for the test in
%and $H_1$ is driven by the functional estimation speed in the
%Sobolev-type space of larger smoothness.\\
%For the problem of adaptive confidence bounds around $f$, this does
%not really matter since $n^{-\frac{s}{2s+1}}$ is precisely of same
%order as the minimax rate of estimation in the sub-model $
%confidence bands without even separating $\Sigma(s,B)$ and $

From the point of view of hypothesis testing, the case $s<2t$ is in
fact of particular interest, as it implicitly addresses the question
whether the complexity of the null hypothesis should influence the
separation rate in non-parametric composite testing problems. When $s
\geq2t$, the rate of estimation in the null hypothesis is of order of
the separation rate, and a reduction to a singleton null hypothesis is
(intuitively) always possible as shown by the infimum test considered
in Bull and Nickl \cite{bull2011adaptive}. For $s<2t$,
new ideas seem to be required.

To the best of our knowledge, the classical literature on
non-parametric hypothesis testing does not answer this question. A
majority of papers consider the case of a \textit{singleton, or a
parametric (finite dimensional) null hypothesis}, see Ingster \cite{ingster1987minimax},
Ingster and Suslina \cite
{ingster2002nonparametric}, Spokoiny \cite
{spokoiny1996adaptive}, Lepski and Spokoiny \cite
{lepski1999minimax}, Horowitz and Spokoiny \cite
{horowitz2001adaptive},
Pouet \cite{pouet2002test}, Fromont and Laurent \cite{fromont2002}. In this case, the null hypothesis is
reducible to a finite union of singletons. The papers that do \textit
{not} consider the case of a simple null hypothesis, such as
D{\"u}mbgen and Spokoiny \cite{dumbgen2001multiscale},
Juditsky and Nemirovski \cite
{juditsky2002nonparametric}, Baraud \textit{et al.} \cite
{baraud2005testing},
consider settings where it is provable that the separation rate $\rho
_n$ must be of the same order as the estimation rate in the alternative
hypothesis ($\rho_n \simeq n^{-t/(2t+1)}$ up to some $\log(n)$ factor).
In particular the gap between estimation and testing rate from which
the problem studied in the present paper arises does not exist, and
plug-in tests that are based on the distance between an estimate of the
function and the null hypothesis, are optimal in these cases. Blanchard \textit{et al.} \cite{blanchard2011testing}
consider a general multiple testing problem
where they test a continuum of null hypotheses. As in Bull and Nickl \cite{bull2011adaptive}, their separation rate
depends on the complexity of
the null hypothesis.
The papers \cite{gine2010confidence} and \cite{nickl2012confidence} consider a composite a non-parametric testing problem, and an approach
based on an infimum test. For the same reason as in the paper \cite{bull2011adaptive}, the complexity of the null
hypothesis affects the separation rates they obtain. Finally the papers \cite{belenik2008,DzieCmiel2014} consider directly
the problem of smoothness testing (or smoothness estimation for \cite{DzieCmiel2014}). However, their perspective
is different and the assumptions they make are very restrictive (for instance, piecewise smoothness, see \cite{DzieCmiel2014}).

%though, it matters to know the answer to this question. At a more high
%level, and after thinking about existing results for signal detection
%and confidence bands, it seems that the fact that there is a
%fundamental difference between the tests described in Equations
%hypothesis implies the difference in the order of the separation $
%considered in paper \citep{bull2011adaptive} is a so-called
%between $f$ and $\Sigma(s,B)$ by the infimum over any element $g$ of $
%concentration of a such statistic is a hard task since $\Sigma(s,B)$
%is an infinite dimensional object. In order to still control this
%infimum test statistic, one has to cover $\Sigma(s,B)$ with $L_2$balls
%and then use a chaining argument. For this reason, the covering number
%of $\Sigma(s,B)$ appears, and it implies that we need $\rho_n \geq D
%n^{-\frac{s}{2s+1}}$.

%More generally, in infinite dimensional testing problems, if one uses
%in a similar way tests that look like infimum tests, then one will
%observe that the complexity of the null hypothesis appears. One can
%then wonder if it really has to be so or not.
In this paper, we demonstrate that the complexity of the null
hypothesis does not influence the separation rate at least in the
testing problem \eqref{test}. % an answer for a such testing problem,
%i.e. the test described in Equation \eqref{test}.
%We prove, for this problem, that the complexity of the null hypothesis
%does \textit{not} appear in the order of the minimal separation $
More precisely, we prove that it is possible to build a test that is
uniformly consistent with a separation rate
\begin{eqnarray*}
\rho_n \simeq n^{-\afrac{t}{2t+1/2}}.
\end{eqnarray*}
%
%which is actually of same order as the $\rho_n$ in the case of the
%testing problem described in Equation \ref{test2}, i.e. where the null
%hypothesis is a point.\\
%The test we propose uses the geometric structure of Sobolev-type
%balls, and in fact is relatively easy to implement (at least compared
%to an infimum test).
The test we propose uses the geometric structure of the Sobolev-type
balls combined with a simple multiple testing idea, and is
straightforward to implement. Our proofs rely on the specific structure
of this problem, and in general whether or not the complexity of $H_0$
influences the separation rate depends heavily on the problem at hand.
%%(\citep{juditsky2002nonparametric} prove that for a different null
%hypothesis, also in $L_2$ norm, the separation rate ).

Section~\ref{sec:setting} formalises the setting and notations that we
consider. Section~\ref{sec:testprob} provides the test and the main
Theorems. Proofs are given in Sections~\ref{sec:proof} and \ref{sec:lowbound}.

%s2 #&#
\section{Setting}\label{sec:setting}

Denote by $L_2([0,1]) = L_2$ the space of functions defined on $[0,1]$
such that $\|f\|_2^2 = \int_0^1 |f(x)|^2 \,\mathrm{d}x < +\infty$, where $\|\cdot\|_2$
is the usual $L_2$ norm. For any functions $(f,g) \in L_2$, we consider
the usual scalar product $\langle f,g \rangle= \int_0^1 f(x) g(x) \,\mathrm{d}x$.

%s2.1 #&#
\subsection{Wavelet basis}

Let $S\geq0$. We consider the Cohen--Daubechies--Vial wavelet basis on
$[0,1]$ with $S$ first null moments (see Cohen \textit{et al.} \cite
{daub1993}), that we write
\begin{eqnarray*}
\{\phi_{k}, k \in Z_{J_0}, \psi_{l,k}, l >
J_0, l\in\mathbb N, k \in Z_l \},
\end{eqnarray*}
where $J_0 \equiv J_0(S)\in\mathbb N^*$ is a constant that grows with
$S$ (see Cohen \textit{et al.} \cite{daub1993}), where $\forall l
\geq J_0, Z_l \subset
\mathbb Z$, and where $\forall k' \in Z_{J_0}$, $\forall l >J_0,
\forall k \in Z_l$, $\phi_{k'}$ and $\psi_{l,k}$ are functions from
$[0,1]$ to $\mathbb R$.

The Cohen--Daubechies--Vial wavelet basis is an orthonormal basis of
functions on $[0,1]$. It is also such that
\begin{eqnarray*}
\forall l \geq J_0+1, \qquad |Z_l| = 2^l\quad
\mbox{and}\quad  z_0 \equiv Z_0(s) = |Z_{J_0}|<
\infty,
\end{eqnarray*}
where $\forall l \geq J_0$, $|Z_l|$ is the number of elements in the
set $Z_l$. Note that the constant $z_0$ grows with $S$ in the
definition of the Cohen--Daubechies--Vial wavelet basis, and is such that
$z_0 \geq1$. We write $\forall k \in Z_{J_0}, \psi_{J_0,k} = \phi_k$
in order to simplify notations. %The Cohen-Daubechies-Vial basis
%verifies all these conditions, see \citep{daub1993}.

%regularity doesn't exist, you can't have l starting at one AND
%|Z_l|=2^l. To use the CDV wavelets you need to start at some high
%enough resolution level J _0 (see for instance the book chapter I've
%written, or the paper with Adam). This is of course only notation but
%you have to be precise here!}

For any function $f \in L_2$, we consider the sequence $a \equiv a(f)$
of coefficients such that $\forall l\geq J_0, \forall k \in Z_l$,
\begin{eqnarray*}
a_{l,k} = \int_0^1
\psi_{l,k}(x) f(x) \,\mathrm{d}x = \langle\psi_{l,k},f \rangle.
\end{eqnarray*}
The functions $f \in L_2$ have the representation
%
%e2.1 #&#
\begin{equation}
\label{eq:sequfunc} f = \sum_{l\geq J_0} \sum
_{k \in Z_l} \psi_{l,k} \langle\psi_{l,k},f
\rangle= \sum_{l\geq J_0} \sum_{k \in Z_l}
a_{l,k} \psi_{l,k}.
\end{equation}
We moreover write for any $J\geq J_0$
\begin{eqnarray*}
\Pi_{V_J}(f) = \sum_{J_0\leq l\leq J} \sum
_{k \in Z_l} a_{l,k} \psi_{l,k}
\end{eqnarray*}
the projection of $f$ onto $V_J = \operatorname{span}(\psi_{l,k},
J_0\leq l \leq J, k \in Z_l)$ (where for any $A \subset L_2$,
$\operatorname{span}(A)$ is the vectorial sub-space generated by the
functions in $A$). We also write
\begin{eqnarray*}
\Pi_{W_J}(f) = \sum_{k \in Z_J} a_{J,k}
\psi_{J,k}
\end{eqnarray*}
the projection of $f$ onto $W_J = \operatorname{span}(\psi_{J,k}, k
\in Z_J)$.

%s2.2 #&#
\subsection{Besov spaces}

We consider, for $r>0$, the $(r,2,\infty)$-Besov (Nikolskii) norms
\begin{eqnarray*}
\|f\|_{r,2,\infty} = \sup_{l\geq J_0} \bigl(2^{lr}\bigl|
\langle f,\psi _{l,\cdot}\rangle\bigr|_{l_2} \bigr),
\end{eqnarray*}
where $|u|_{l_2} = (\sum_{i} u_i^2)^{1/2}$ is the sequential $l_2$
norm, and $l_2$ is the associated sequential space.
%Note that $\|.\|_{0,p} \leq c_1 \|.\|_p$ (see \citep{Ingster}).

The associated $(r,2,\infty)$-Besov (Nikolskii) spaces are defined as
\begin{eqnarray*}
B_{r,2,\infty} = \bigl\{f \in L^2\dvtx  \|f\|_{r,2,\infty}<+\infty
\bigr\}.
\end{eqnarray*}

We write for a given $r>0$ and a given $B>0$ the $B_{r,2,\infty}$ Besov
ball of smoothness $r$ and radius $B$ as
\begin{eqnarray*}
\Sigma(r,B) := \bigl\{f\in B_{r,2,\infty}\dvtx  \|f\|_{r,2,\infty}<B\bigr\}.
\end{eqnarray*}

Since the wavelet basis we considered to build the $(r,2,\infty)$-Besov
spaces is the Cohen--Daubechies--Vial wavelets with $S$ first null
moments, the defined $(r,2,\infty)$ Besov spaces correspond to the
functional $(r,2,\infty)$-Besov spaces (Sobolev-type spaces) for any $r
\leq S$, see Meyer \cite{meyer1992wavelets} and
H{\"a}rdle \textit{et al.} \cite{hardle1998wavelets}. %We assume
%that our basis verifies this property with $s \leq S$ where $s$ is the
%largest smoothness that we wish to consider in our testing
%problem.%This is the justification of why it makes sense to work with
%the spaces $B_{s,2,\infty}$.

\begin{Remark*} We chose to consider the Cohen--Daubechies--Vial wavelet
basis for simplicity and clarity in presentation, but any orthonormal
wavelet basis that is such that (i) the number of wavelets $|Z_l|$ at
each level $l$ is bounded by a constant time $2^l$ and (ii) the basis
can be used to characterize the functional $(r,2,\infty)$-Besov spaces
(Sobolev-type spaces), could have been used.
\end{Remark*}

%: this redefinition of the problem is the same as the one in paper
%Sobolev spaces, since they are defined in terms of the wavelet
%coefficients.

%Lofstrom and Besov \textit{et al.} cannot possibly contain anything about BEsov
%spaces and wavelets, as they are written in the 70s. You can cite
%Haedle et al, and perhaps also the book by Y. Meyer.}

%s2.3 #&#
\subsection{Observation scheme}

Let $n>0$. The data is a realisation of a Gaussian process defined for
any $x \in[0,1]$ as%and for a given $n>0$ as
\begin{eqnarray*}
\mathrm{d}Y^{(n)}(x) = f(x) \,\mathrm{d}x + \frac{\mathrm{d}B_x}{\sqrt{n}},
\end{eqnarray*}
where $(B_x)_{x\in[0,1]}$ is a standard Brownian motion, and $f \in
L_2$ is the function of interest.

Let us write for any $l \geq J_0$ and $k \in Z_l$ the associated
wavelet coefficients as
\begin{eqnarray*}
\hat a_{l,k}= \bigl\langle\psi_{l,k},\mathrm{d}Y^{(n)} \bigr
\rangle= \int_0^1 \psi _{l,k}(x)f(x)\,\mathrm{d}x
+ \frac{1}{\sqrt{n}}\int_0^1
\psi_{l,k}(x)\,\mathrm{d}B_x, \quad \mbox {and}\quad  a_{l,k}= \langle
\psi_{l,k},f \rangle,
\end{eqnarray*}
where for any $g \in L_2$, $\int_0^1 g(x)\,\mathrm{d}B_x$ is the usual stochastic
integral, and is as such distributed as a Gaussian random variable of
mean $0$ and variance $\|g\|_2^2$. Since the Cohen--Daubechies--Vial
wavelet basis is orthonormal, the coefficients $(\hat a_{l,k})_{l\geq
J_0,k \in Z_l}$ are jointly Gaussian random variables such that
\begin{eqnarray*}
(\hat a_{l,k})_{l\geq J_0,k \in Z_l} \sim\mathcal N \biggl((a_{l,k})_{l\geq J_0,k \in Z_l},
\biggl(\frac{1}{n}\mathbf1\bigl\{l=l', k=k'
\bigr\} \biggr)_{l\geq J_0,k \in Z_l,l'\geq J_0,k'\in Z_{l'}} \biggr),
\end{eqnarray*}
where $\mathcal N (\mu, \sigma^2)$ is the normal distribution of mean
$\mu$ and variance-covariance $\sigma^2$ (and where we write $X \sim
\mathcal N (\mu, \sigma^2)$ for stating that $X$ is such a Gaussian
distribution) and where $\mathbf1\{\,\cdot\,\}$ is the usual indicator function.

We consider the wavelet estimate of $f$:
\begin{eqnarray*}
\hat f_n =\sum_{l\geq J_0} \sum
_k \hat a_{l,k} \psi_{l,k}.
\end{eqnarray*}
This estimate is of infinite variance in $L_2$, hence projected estimates
\begin{eqnarray*}
\hat f_n(j) := \Pi_{V_j} \hat f_n,
\end{eqnarray*}
have to be considered.

In the sequel, we write $\Pr_f$ (respectively $\mathbb E_f$, and
$\mathbb V_f$) the probability (respectively expectation, and variance)
under the law of $\mathrm{d}Y^{(n)}$ when the function underlying the data is
$f$. When no confusion is likely to arise, we write simply $\Pr$
(respectively, $\mathbb E$ and $\mathbb V$).

\begin{Remark*} The spaces $B_{r,2,\infty}$ are slightly larger than
the usual Sobolev spaces,
see Bergh and L{\"o}fstr{\"o}m \cite
{bergh1976interpolation} and Besov \textit{et al.} \cite
{besov1978integral}. They are however
the natural objects to consider for a smoothness test, since they are
the largest Besov spaces where adaptive estimation remains possible
(see Donoho \textit{et al.} \cite{donoho1996density} and
Bull and Nickl \cite{bull2011adaptive}). Indeed, one
can prove
that there exists an estimate $\tilde f_n(Y^{(n)})$ of $f$ such that
for any $S \geq r >1/2$ and $B>0$, we have
\begin{eqnarray*}
\sup_{f \in\Sigma(r,B)} \mathbb E \|\tilde f_n - f
\|_2 \leq \mathrm{O}\bigl(n^{-r/(2r+1)}\bigr),
\end{eqnarray*}
see, for instance, Theorem~2 in the paper Bull and Nickl \cite{bull2011adaptive} (with
some simple modifications needed for the regression situation
considered in the present paper).
\end{Remark*}

%s3 #&#
\section{Testing problem}\label{sec:testprob}

%s3.1 #&#
\subsection{Formulation of the testing problem}

Let $S \geq s >t>0$ (we choose the Cohen--Daubechies--Vial wavelet basis
with $S$ first null moments with $S$ larger than $s$). We want to test
whether $f$ is in $\Sigma(s,B)$, or whether $f$ is outside this ball,
i.e., in $\Sigma(t,B) \setminus\Sigma(s,B)$. This is generally
impossible to do uniformly and functions that are $t$ smooth but too
close from $s$ smooth functions (such that the $L_2$ distance between
these functions and the Sobolev-type ball of smoothness $s$ is small)
have to be removed.

Let us first define the restriction of the sets $\Sigma(t, B)$ to sets
that are separated away from $\Sigma(s,B)$ by some minimal distance
$\rho_n>0$:
\begin{eqnarray*}
\tilde\Sigma(t, B,\rho_n) = \bigl\{f \in\Sigma(t,B)\dvtx  \bigl\|f
- \Sigma (s,B)\bigr\|_2 \geq\rho_n \bigr\},
\end{eqnarray*}
where we remind that for any set $G \subset L_2$, we have $\|f - G\|_2
= \inf_{g \in G}\|f - g\|_2$.

The testing problem is the following
\begin{eqnarray*}
H_0\dvtx  f \in\Sigma(s,B)\quad  \mbox{vs.}\quad  H_1\dvtx  f \in\tilde\Sigma(t,B,
\rho_n).
\end{eqnarray*}
When no confusion is likely to arise, we will use the short-hand
notation $f \in H_0$ for $f \in\Sigma(s,B)$, and $f \in H_1$ for $f
\in\tilde\Sigma(t,B, \rho_n)$.
%where we remind that $\Sigma(s,B)$ is the ball of radius $B$ of
%$B_{s,2,\infty}$ and $\tilde\Sigma(t,B, \rho_n) = \Sigma(t,B) - \{g:

%Paper \citep{bull2011adaptive} provides a method to test between the
%two hypotheses with bounder $\rho_n = O(\max(n^{-\frac{t}{2t+1/2}},
%n^{-\frac{s}{2s+1}}))$, and uses this method to build confidence bands
%around the function. Paper \citep{ingster1987minimax} proves that the
%minimal width of the bounder is $\rho_n= \Omega(n^{-
%possible to make these two bounder sizes match.

%This is in all generality impossible and it is proved in e.g.
%uniformly feasible, one has to remove functions of smoothness $t$ that
%are too close in $L_2$ norm to functions of smoothness $s$. The paper
%diameter of the ball to remove

%s3.2 #&#
\subsection{Main results}\label{ss:stats2}

Let $j \geq J_0$ such that $j = \lfloor1/(2t+1/2) \log(n)/\log
(2)\rfloor$, where $\lfloor\,\cdot\,\rfloor$ is the integer part of a
real number. In particular, this definition implies that
$n^{1/(2t+1/2)}/2 \leq2^j \leq n^{1/(2t+1/2)}$.

Consider for any $J_0 < l\leq j$ the test statistics
%
%e3.1 #&#
\begin{equation}
\label{eq:Tnl} T_n(l)= \|\Pi_{W_{l}} \hat f_n
\|_{2}^2 - \frac{2^l}{n}, \quad \mbox{and}\quad
T_n(J_0)= \|\Pi_{W_{J_0}} \hat f_n
\|_{2}^2 - \frac{z_0}{n}.
\end{equation}

These quantities $T_n(l)$ are estimates of $\|\Pi_{W_{l}} (f) \|_{2}^2$
across all levels $J_0 \leq l \leq j$. Concerning levels $l>j$, even in
the worst case of smoothness $t$, the $L_2$ norm of the function at
these levels is smaller than $n^{-t/(2t+1/2)}$, that is, $\|f - \Pi
_{V_j}(f)\|_2 =\mathrm{O}(n^{-t/(2t+1/2)})$. This implies that one does not need
to control for what happens at these levels.

Let $\alpha>0$ be the desired level of the test. %large enough (but
%depending only on $B, s, t$ and the desired level of the test $
Consider the positive constants $t_n(l)$ such that for any $J_0 \leq l
\leq j$
%
%e3.2 #&#
\begin{equation}
\label{def:Tnl} t_n(l)^2 = \biggl(\frac{B}{2^{ls}} +
\frac{\tau_{l}}{2} \biggr)^2 = \frac
{B^2}{2^{2ls}} + \frac{B}{2^{ls}}
\tau_l + \frac{\tau_l^2}{4},
\end{equation}
where the sequence $(\tau_l)_{J_0\leq l \leq j}$ is such that for any
$J_0 < l \leq j$
%
%e3.3 #&#
\begin{equation}
\label{def:Tl} \tau_l \equiv\tau_{n,l} = 24\sqrt{
\frac{z_0}{\alpha}} \frac
{2^{(j+l)/8}}{\sqrt{n}}, \quad \mbox{and}\quad  \tau_{J_0} \equiv
\tau_{n,J_0} = 24\sqrt{\frac{z_0}{\alpha}} \frac{1}{\sqrt{n}}.
\end{equation}
%
%where $C'\geq1$ is a universal constant.

%for $C(\alpha) \geq8(1+\frac{1}{B})\sqrt{\frac{\max(1,C')}{\alpha}}$,
%where $C'$ is some universal constant.

We consider the test:
\begin{eqnarray*}
\Psi_n(\alpha) = 1 - \prod_{J_0 \leq l \leq j}
\mathbf1 \bigl\{T_n(l) < t_n(l)^2\bigr\},
\end{eqnarray*}
where we remind that $\mathbf1\{\,\cdot\,\}$ is the usual indicator
function. We reject $H_0$ as soon as the test statistic at one of the
levels $J_0 \leq l\leq j$ indicates a too large Besov norm. The
intuition behind this test is that $f$ belonging to $\Sigma(s,B)$ is
equivalent to $\|\Pi_{V_l}(f)\|_{s,2,\infty}$ being smaller than or
equal to $B$ for any $l\geq J_0$. As explained before, we do not need
to be too concerned by what happens for $l>j$. In the case $J_0 \leq
l\leq j$, each statistic $T_n(l)$ is designed to test this. We
illustrate this in Figure~\ref{fig:expl}.
%
%f1 #&#
\begin{figure}

\includegraphics{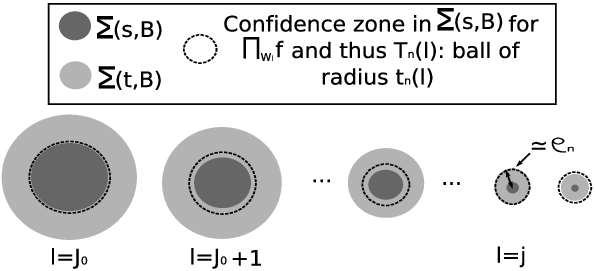}

\caption{Illustration of the testing problem.} \label{fig:expl}
\end{figure}

%for Haar wavelets and
% \Psi_n' = 1 - \ind{T_n'<(\tau_n')^{2}} \prod_{l \leq j_0} \ind{\tilde
%T_n^{(l)}<\tilde\tau_n^{(l)}}

%We will prove formally that this test is consistent in the sequel.

We provide the following definition of consistency for a test,
following the line of work of Ingster and Suslina
\cite{ingster2002nonparametric}.

%de3.1 #&#
\begin{defn}[($\alpha$-consistency)]\label{def:consist}
Let $\alpha>0$ and $H_0, H_1$ be two hypotheses (functional sets). Let
$\Psi_n(Y^{(n)},H_0, H_1, \alpha)$ be a test, that is to say a
measurable function taking values in $\{0,1\}$. We say that $\Psi_n$ is
$\alpha$-consistent if we have for any $n>0$
\begin{eqnarray*}
\sup_{f \in H_0} \mathbb E_f \Psi_n + \sup
_{f \in H_1} \mathbb E_f (1-\Psi_n) \leq
\alpha.
\end{eqnarray*}
\end{defn}

%The intuition behind this test is as follows. If one projects $f_n$
%(and thus $f$) on the span generated by wavelets of large resolution
%(i.e. $V_{j_0}$), then the variance of the $\|.\|_{s,2,\infty}$ norm
%of this projection is small (smaller than $2^{2ls}\rho_n^2 = o(1)$ if
%we project for all $l$) and one can then test if this norm is larger
%than $B$ or not. If one projects $f_n$ (and thus $f$) on the span
%generated by wavelets of small resolution (i.e. $V_j - V_{j_0}$), then
%$f$ belonging to $\Sigma(s,B)$ means that the $\|.\|_2$ norm of $f$ is
%small (smaller than $\rho_n$), and the testing problem that we
%consider is equivalent to testing the nullity of $f$.

We now state the main result of this paper.%, i.e. that the test we
%presented is consistent for $\rho_n = 2BCn^{-t/(2t+1/2)}$, for $C$
%large enough (but depending only on $ B, s, t$ and $\alpha$).

%th1 #&#
\begin{theorem}\label{thm:consist}
Let $\alpha>0$. The test $\Psi_n(\alpha)$ is an $\alpha$-consistent
test for discriminating between $H_0$ and $H_1$ and for $\rho_n =
\tilde C(\alpha)n^{-t/(2t+1/2)}$, where $\tilde C(\alpha) = 24(\frac
{2^tB}{\sqrt{1 - 2^{-2t}}}+19)\sqrt{\frac{1}{\alpha}}$.%, for $C'
%1$ some universal constant.
\end{theorem}

The proof of this theorem is in Section~\ref{sec:proof}. The region we
had to remove so that $\Psi_n(\alpha)$ is $\alpha$-consistent could not
have been taken significantly smaller, as stated in the next theorem.
%This implies in particular that $\Psi_n(\alpha)$ is a minimax optimal
%test.

%th2 #&#
\begin{theorem}\label{th:lb}
Let $1>\alpha\geq0$. There exists no $\alpha$-consistent test for
discriminating between $H_0$ and $H_1$ and for $\rho_n = \tilde
D(\alpha
)n^{-t/(2t+1/2)}$, where $\tilde D(\alpha) = \min ( (\frac
{1-\alpha}{2} )^{1/4}, B )$.% is a constant that depends only
%on
%$B$ and the level $\alpha$ of the test.%, and that is small enough.
\end{theorem}

The proof of Theorem~\ref{th:lb} is in Section~\ref{sec:lowbound}. It
is very similar to the proofs in papers Ingster \cite{ingster1987minimax} and
Bull and Nickl \cite{bull2011adaptive} (the proof in
paper Bull and Nickl \cite{bull2011adaptive} holds in
the more involved case of density estimation).

We would like to emphasise that the test $\Psi_n$, in addition to being
rather simple conceptually, is quite easy to implement since it
requires only the computation of (significantly) less than $n$
integrals/sums -- the empirical coefficients -- and less than $\log(n)$
sums of squares of these coefficients. It can replace the more
complicated infimum test considered in the paper Bull and Nickl \cite{bull2011adaptive} for the creation of
adaptive and honest confidence
bands.%It has indeed the advantages of the test described in

%s3.3 #&#
\subsection{Alternative settings}

We provided in the last subsection a consistent test on a model that
could not have been taken significantly larger. This test was
constructed in the rather simplistic setting of non-parametric Gaussian
homoscedastic regression with normalised variance. %There are several
%comments which can be made concerning this test, among which we
%distinguish two main directions. The first concerns extending the
%results that we have in this toy-setting to other more realistic
%situations, and the other concerns the application of these results to
%a \textit{computationally simple} construction of confidence bands.
But in many cases (see, e.g., Rei{\ss} \cite
{reiss2008asymptotic} and  Nussbaum \cite
{nussbaum1996asymptotic}), it has been proven that
it generalises rather well to more realistic and complex settings. The
concern in our case, however, is that we heavily rely on the \textit
{homoscedasticity} assumption with known variance of the noise. Indeed,
we subtract the constant part induced by this variance in the estimates
of $T_n(l)$ in equation \eqref{eq:Tnl}. This part is much larger than
the deviations (in high probability) of $\| \Pi_{W_l} \hat f_n\|_2^2$
around its mean, and it is thus crucial to remove it. We illustrate
this in Figure~\ref{fig:noi}.

There is however a way around this problem that we discuss now, as well
as generalizations to more complex settings.
%
%f2 #&#
\begin{figure}[b]

\includegraphics{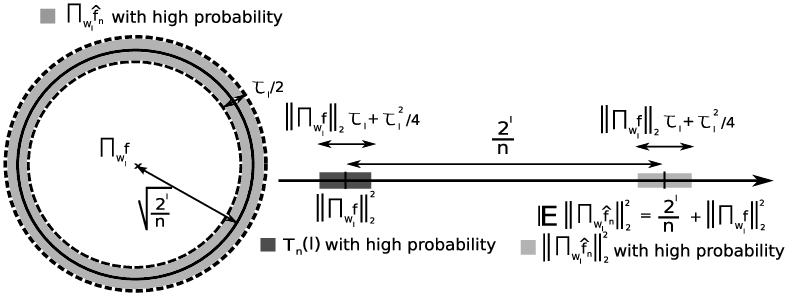}

\caption{Statistics $ T_n(l)$ and the removal of the expectation of the
square of the expectation of the noise.} \label{fig:noi}
\end{figure}

\emph{Heteroscedastic non-parametric Gaussian regression}. Assume now
that the data are generated according to the process
\begin{eqnarray*}
\mathrm{d}Y^{(n)}(x) = f(x) \,\mathrm{d}x + \frac{\sigma(x) \,\mathrm{d}B_x}{\sqrt{n}},
\end{eqnarray*}
where $(B_x)_{x\in[0,1]}$ is a standard Brownian motion, and $f,
\sigma
\in L_2$. Since the function $\sigma$ is unknown, we cannot apply the
technique we described. However, if we know a upper bound on $\|\sigma
\|
_2$, it is still possible to solve this problem with a very similar
technique.

The modification goes as follows. We start by dividing the initial
sample in two sub-samples of equal size $n/2$. Then we compute the
empirical estimates of the function in these two samples and write
$\hat f_n^{(1)}$ and $\hat f_n^{(2)}$ for the estimates of the function
computed in each of the two halves. We then define the statistics $\hat
T_n(l)$ (which play the same role as the $T_n(l)$) as
%
%e3.4 #&#
\begin{equation}
\label{eq:alterTnl} \hat T_n(l) = \bigl\langle\Pi_{W_l} \hat
f_n^{(1)} , \Pi_{W_l} \hat f_n^{(2)}
\bigr\rangle.
\end{equation}
Since $\hat f_n^{(1)}$ and $\hat f_n^{(2)}$ are independent estimates
of $f$, the additional term that comes from the expectation of the
square of the noise (the variance) disappears and it is possible to
prove that this newly defined $\hat T_n(l)$ concentrates around $\|\Pi
_{W_l} f\|_2^2$ with an error of same order as in Lemma~\ref
{lem:conc0pp} below. This implies that we can test in a similar way and
derive similar results.

\emph{Regression, density estimation and autoregressive model}. The
settings of non-parametric regression (with noise that can be
non-Gaussian), of non-parametric density estimation, and of
non-parametric auto-regressive model ($\AR(1)$) are not too different
from the heteroscedastic setting under a given set of assumptions (that,
e.g., the noise on the data is sub-Gaussian and that the design is
adapted for regression, and that, e.g., the regression function/density
is\vadjust{\goodbreak} bounded, see Bull and Nickl \cite
{bull2011adaptive}. This follows from the
asymptotic equivalence between these models and non-parametric Gaussian
regression (again, see, e.g., Rei{\ss} \cite
{reiss2008asymptotic} and Nussbaum \cite
{nussbaum1996asymptotic}).
\begin{itemize}[$\bullet$]
\item[$\bullet$] In the regression setting, we assume that the $n$ data $(X_i,
Y_i)_{i\leq n}$ are
\begin{eqnarray*}
Y_i = f(X_i) + \sigma(X_i)
\epsilon_i,
\end{eqnarray*}
where $\epsilon_i$ are independent random variables of mean $0$ and
variance $1$. Based on these data, we can compute also estimates for
the wavelet coefficients of $f$ as
\begin{eqnarray*}
\hat a_{l,k} = \frac{1}{n} \sum_{i=1}^n
Y_i\psi_{l,k}(X_i),
\end{eqnarray*}
and thus estimate $f$. Then we can follow the procedure described in
the setting of heteroscedastic non-parametric Gaussian regression
(equation \eqref{eq:alterTnl}). However, one needs to be careful in
this setting since the design (i.e., position of the $X_i$) is crucial.
Indeed, wavelets are very localised functions and estimating the
wavelet coefficients in a reasonably accurate way requires that the
points $X_i$ are spread over the whole domain, that is to say that
there are enough points in each region of the domain. In particular, a
standard random design will fail in this case, see H{\"a}rdle \textit{et al.} \cite{hardle1998wavelets}.
\item[$\bullet$] In the density estimation setting, we assume that the $n$ data
generated by $f$ are $(X_i)_i$, and estimate the wavelet coefficients
of $f$ as
\begin{eqnarray*}
\hat a_{l,k} = \frac{1}{n} \sum_{i=1}^n
\psi_{l,k}(X_i),
\end{eqnarray*}
and thus estimate $f$. Then we can follow the procedure described in
the setting of heteroscedastic non-parametric Gaussian regression
(equation \eqref{eq:alterTnl}).
\item[$\bullet$] We consider finally the non-parametric autoregressive model with
memory $1$ (or $\AR(1)$). The output $(X_i)_{i \leq n}$ of an $\AR(1)$
can be described as follows:
\begin{eqnarray*}
X_{i+1} = f(X_i) + \sigma(X_i)
\epsilon_i.
\end{eqnarray*}
After sub-sampling the data at random in order to make them close to
independent, one can go back to the regression setting, and apply the
same method (see, e.g., Hoffmann \cite
{hoffmann1999nonparametric} for
equivalence of this setting and regression setting after sub-sampling).
\end{itemize}

%An interesting feature of a such test is that it is (and also it
%remains, in all the variant described in the past Subsection) easy to
%implement. Indeed, it requires only the computation of a serie of less
%than $n$ integrals/sums (the estimates of the coefficients), and then
%then the computation of less than $\log(n)$ sums of squares. We
%believe that this can be exploited to compute confidence bands around
%the function in a computationally simpler way than what already
%exists. In $L_2$, paper \citep{bull2011adaptive} proposes also a such
%test, but that is computationally more involved, since it relies on
%minimising the distance between the estimate of the function and a
%Sobolev ball, i.e. an infinite dimensional object. We proved in this
%paper that it is possible to avoid that, and suggest thus, for more
%easy computations, to implement this test instead of the infimum test
%in paper \citep{bull2011adaptive}, when one wishes to construct
%confidence bands.

%white noise? I think you should elaborate here to explain what is
%going on. -) the remarks on p.9 in regression need to be made more
%careful: one can only use wavelets if the design is sufficiently
%equally spaces, particularly with random design wavelets will fail
%(see the discussion in haerdle \textit{et al.}.), as you need sufficiently many
%points everywhere for the wavelet approximation of the regression
%function to be good enough.}

%s4 #&#
\section{Proof of Theorem \texorpdfstring{\protect\ref{thm:consist}}{3.1}}\label{sec:proof}

This section contains a proof of Theorem~\ref{thm:consist}.

%s4.1 #&#
\subsection{Decomposition of the problem}

The statistics $T_n(l)$ are unbiased estimates of $\|\Pi_{W_{l}}(f)\|
_{2}^2$ for any $J_0 \leq l \leq j$, as explained later in this
section. Assuming this, the next lemma explains why the test $\Psi_n$
that we described is a reasonable thing to do.\vadjust{\goodbreak}

%le3 #&#
\begin{lem}\label{lem:justiftest}
Let $(\tau_l)_{J_0 \leq l\leq j}$ be a sequence of positive real
numbers. Assume that
\begin{eqnarray*}
\rho_n\geq\biggl(4 \frac{B}{\sqrt{1 - 2^{-2t}}} 2^{-jt} + \frac{4}{3} \sum
_{J_0\leq
l\leq j} \tau_l \biggr).
\end{eqnarray*}
Then we have
\begin{itemize}[$\bullet$]
\item[$\bullet$]$f \in H_0 \Rightarrow\max_{J_0 \leq l\leq j}  (\|\Pi
_{W_{l}}(f)\|_{2} - \frac{B}{2^{ls}} ) \leq0$.
\item[$\bullet$]$f \in H_1 \Rightarrow\max_{J_0 \leq l\leq j} (\|\Pi
_{W_{l}}(f)\|_{2} - \frac{B}{2^{ls}} - \tau_l  ) > 0$.
\end{itemize}
%
%Note that if $j$ and $j_0$ are such that $2^j \approx n^{-1/(2t+1/2)}$
%and $2^{j_0} \approx n^{-\frac{t}{s(2t+1/2)}}$, the condition on $
%minimax-optimal boundary condition.
\end{lem}

\begin{pf}
\emph{Under the null Hypothesis $H_0$}.
If $f$ is in $\Sigma(s,B)$, then by definition of the Besov spaces
\begin{eqnarray*}
\|\Pi_{V_{j}} f\|_{s,2,\infty}\leq B,
\end{eqnarray*}
which implies by definition of the $\|\cdot\|_{0,2,\infty}$ norm that
\begin{eqnarray*}
\sup_{J_0\leq l\leq j} \biggl(\|\Pi_{W_l} f \|_{0,2,\infty} -
\frac{B
}{2^{ls}} \biggr) \leq0.
\end{eqnarray*}
This implies by Parseval's identity, and since $\|\Pi_{W_l} f \|
_{0,2,\infty} =\|\Pi_{W_l} f \|_{2}$
\begin{eqnarray*}
\sup_{J_0\leq l\leq j} \biggl(\|\Pi_{W_l} f \|_{2} -
\frac{B
}{2^{ls}} \biggr) = \sup_{J_0\leq l\leq j} \biggl(\|
\Pi_{W_l} f \|_{0,2,\infty} - \frac{B
}{2^{ls}} \biggr) \leq0.
\end{eqnarray*}

\emph{Under the alternative Hypothesis $H_1$}.
Assume that $f$ is in $\tilde\Sigma(t, B, \rho_n)$.
%We write
%for $g^*(f) = \arg\min_{g\in\Sigma(s,B)} \|f - g\|_{2}$ (it exists
%since $\Sigma(s,B)$ is a compact). More generally, we note in the
%sequel $g^*(u)$ for the minimiser in $u$.
By triangular inequality, we have
\begin{eqnarray*}
\inf_{g \in\Sigma(s,B)} \|f - g\|_{2} %&\leq\inf_{g \in\Sigma(s,B)}
&
\leq&\inf_{g \in\Sigma(s,B)}\bigl\|\Pi_{V_{j}} (f) - g\bigr\|_{2} +
\bigl\|f - \Pi _{V_j} (f)\bigr\|_{2}
\\
&\leq&\inf_{g \in\Sigma(s,B)} \bigl\|\Pi_{V_{j}} (f) - g
\bigr\|_{2} + \frac
{B}{\sqrt{1 - 2^{-2t}}}2^{-jt},
\end{eqnarray*}
since by definition of the $(t,2,\infty)$ Besov space, we know that
\begin{eqnarray*}
\bigl\|f - \Pi_{V_j} (f)\bigr\|_{2} \leq\sqrt{\sum
_{l=j+1}^{\infty} 2^{-2lt} B^2} \leq
\frac{B}{\sqrt{1 - 2^{-2t}}} 2^{-jt}.
\end{eqnarray*}
We thus have, since $\rho_n \leq\inf_{g \in\Sigma(s,B)} \|f - g\|
_{2}$ by definition of $\tilde\Sigma(t, B, \rho_n)$, and since $\rho_n
\geq(4 \frac{B}{\sqrt{1 - 2^{-2t}}} 2^{-jt} + 4/3 \sum_{J_0\leq
l\leq
j} \tau_l )$
%
%e4.1 #&#
\begin{eqnarray}
\label{eq:inequa} 3\rho_n/4 \leq\rho_n -
\frac{B}{\sqrt{1 - 2^{-2t}}}2^{-jt} \leq \inf_{g \in\Sigma(s,B)} \bigl\|
\Pi_{V_{j}} (f) - g\bigr\|_{2}.
\end{eqnarray}

%Also, by triangular inequality,
%since for any function $h$, and any $l\geq0$, $\|\Pi_{V_{l} -
%V_{l-1}} (h)\|_2 = \|\Pi_{V_{l} - V_{l-1}} (h)\|_{0,2, \infty}$.

Let us write $(a_{l,k})_{l,k}$ the coefficients of $f$ and
$(b_{l,k})_{l,k}$ the coefficients of the minimiser $g$. We have by
definition of $\Sigma(s,B)$, by the triangular inequality and by
Parseval's identity
\begin{eqnarray*}
&&\inf_{g \in\Sigma(s,B)}\bigl\|\Pi_{V_{j}} (f) -g\bigr\|_{2} \\
&&\quad \leq
\inf_{g \in\Sigma(s,B)} \sum_{l=J_0}^{j}
\bigl\|\Pi_{W_{l} } (f) - g\bigr\| _{2}
\\
&&\quad =\inf_{(b_{l,k})_{l,k} : \forall l \geq J_0, 2^{ls}\|b_{l,\cdot}\|_{l_2}
\leq B} \Biggl(\sum_{l=J_0}^{j}
\sqrt{\sum_{k\in Z_l} (a_{l,k} -
b_{l,k})^2} + \sum_{l=j+1}^{\infty}
\sqrt{\sum_{k\in Z_l} b_{l,k}^2}
\Biggr)
\\
%&= \inf_{(b_{l,k})_{l,k} : \forall l\geq J_0, 2^{ls}\|b_{l,.}\|_{l_2}
&&\quad = \sum_{l=J_0}^{j}
\inf_{(b_{l,k})_{l,k} : 2^{ls}\|b_{l,\cdot}\|_{l_2}
\leq B} \sqrt{\sum_{k\in Z_l}
(a_{l,k} - b_{l,k})^2},
\end{eqnarray*}
since the constraints defining the minimisation problems involved do
not interact across the levels~$l$. The last equation, together with
equation \eqref{eq:inequa}, implies that
\begin{eqnarray*}
3\rho_n/4 \leq\sum_{l=J_0}^{j}
\inf_{(b_{l,k})_{l,k} : 2^{ls}\|
b_{l,\cdot}\|_{l_2} \leq B} \sqrt{\sum_{k\in Z_l}
(a_{l,k} - b_{l,k})^2}.
\end{eqnarray*}
By definition, $\rho_n \geq4/3 \sum_{l=J_0}^{j} \tau_l$, so the last
equation implies that
%
%e4.2 #&#
\begin{eqnarray}
\label{eq:lastequhihi} \sum_{l=J_0}^{j}
\tau_l \leq\sum_{l=J_0}^{j} \inf
_{(b_{l,k})_{l,k} :
2^{ls}\|b_{l,\cdot}\|_{l_2} \leq B} \sqrt{\sum_{k\in Z_l}
(a_{l,k} - b_{l,k})^2}.
\end{eqnarray}

At least one of the $\tau_l$'s has to be less than or equal to
\begin{eqnarray*}
\inf_{(b_{l,k})_{l,k} : 2^{ls}\|b_{l,\cdot}\|_{l_2} \leq B} \sqrt{\sum_{k\in Z_l}
(a_{l,k} - b_{l,k})^2},
\end{eqnarray*}
as otherwise $\sum_{l=J_0}^{j} \tau_l$ would exceed the right-hand side
in equation \eqref{eq:lastequhihi}. Let $J_0 \leq l\leq j$ be one of
these indexes, we have%So there exists at least an integer $l$ such
%that $J_0 \leq l\leq j$ and also such that $\tau_l \leq
%(a_{l,k} - b_{l,k})^2}$ and thus
%
\begin{eqnarray*}
\tau_l &\leq&\inf_{(b_{l,k})_{l,k} : 2^{ls}\|b_{l,\cdot}\|_{l_2} \leq B} \sqrt{\sum
_{k\in Z_l} (a_{l,k} - b_{l,k})^2}
\\
&\leq&\max \biggl(0,\sqrt{\sum_{k\in Z_l}
a_{l,k}^2} - \frac
{B}{2^{ls}} \biggr)
\\
&\leq&\bigl\|\Pi_{W_l} (f)\bigr\|_{2} - \frac{B}{2^{ls}}
\end{eqnarray*}
since by definition of the Euclidian ball, for any $u \in l_2$, we have
$\inf_{v \in l_2: \|v\|_{l_2} = 1} \|u - v\|_{l_2} = \max(0, \|u\|
_{l_2} - 1)$.

This concludes the proof.
\end{pf}

%s4.2 #&#
\subsection{Convergence tools for $T_n(l)$}

The next lemma is a standard and also rather weak concentration
inequality (see, e.g., Birg{\'e} \cite
{birge2001alternative} for similar
results). %It is far from being optimal in terms of the probability of
%error,

%le4 #&#
\begin{lem}\label{lem:conc0pp}
Let $\Delta>0$. Then %There exists an universal constant $C'\geq1$
%such that
%+ 2^{l/4}\frac{\|\Pi_{W_{l}}f\|_{2}^2}{n} )} \Big\} \leq\Delta.
%
\begin{eqnarray*}
\Pr \biggl\{\forall l\dvtx  J_0 \leq l \leq j, \bigl|T_n(l) - \|
\Pi_{W_l} f\| _{2}^2 \bigr| \geq4\sqrt{
\frac{3z_0}{\Delta} \biggl(\frac{2^{(j+l)/2}}{n^{2}} + 2^{l/4}\frac{\|\Pi_{W_{l}}f\|_{2}^2}{n}
\biggr)} \biggr\} \leq\Delta.
\end{eqnarray*}
\end{lem}

\begin{pf}
Let $J_0 < l\leq j$.
Note first that by Parseval's identity, we have $\|\Pi_{W_{l}} \hat
f_n\|_2^2 = \sum_k \hat a_{l,k}^2$. Then we have by definition $T_n(l)
= \sum_k \hat a_{l,k}^2 - \frac{ 2^l}{n}$.\vspace*{2pt} %Somehow, testing nullity
%of
%$\|\Pi_{V_j - V_{j_s}} f\|_{2}^2$ is equivalent to performing a $
%coefficient is in fact a $\chi_2$ test.

We have $\hat a_{l,k} = a_{l,k} + \hat a_{l,k} - a_{l,k}$ where $\hat
a_{l,k} - a_{l,k} \sim\mathcal N(0,1/n)$ (by assumption of the
Gaussian model), and thus we have
\begin{eqnarray*}
\mathbb E |\hat a_{l,k}|^{2} = \frac{1}{n} +
a_{l,k}^{2}.
\end{eqnarray*}
Also since for any constant $m \in\mathbb R$, and for $G \sim\mathcal N(0,1)$,
\begin{eqnarray*}
\mathbb V (G+m)^2 = \mathbb E \bigl(G^2 + 2Gm -1
\bigr)^2 = \mathbb E \bigl(G^4 + 4G^2m^2
+1 - 2G^2\bigr) = 4m^2 + 2 \leq4\bigl(1 + m^2
\bigr),
\end{eqnarray*}
we have
\begin{eqnarray*}
\mathbb V |\hat a_{l,k}|^{2} \leq4\biggl(\frac{1}{n^2}
+ \frac{a_{l,k}^{2}}{n}\biggr).
\end{eqnarray*}
This implies since the $\hat a_{l,k}$ are independent Gaussian random variables
\begin{eqnarray*}
\mathbb E \biggl( \sum_{k \in Z_l} \hat
a_{l,k}^2 \biggr) = \sum_{k \in
Z_l}
a_{l,k}^2 + \frac{2^l}{n},
\end{eqnarray*}
and
\begin{eqnarray*}
\mathbb V \biggl( \sum_{k \in Z_l} \hat
a_{l,k}^2 \biggr) \leq4 \biggl( \frac
{2^{l}}{n^2} +
\frac{\sum_{k \in Z_l} a_{l,k}^2}{n}\biggr).
\end{eqnarray*}
This implies by Chebyshev's inequality that for any $\delta_l>0$, we have
\begin{eqnarray*}
\Pr \biggl\{ \biggl| \sum_{k \in Z_l} \hat a_{l,k}^2-
\frac{2^l}{n} - \sum_{k \in Z_l} a_{l,k}^2
\biggr| \geq\sqrt{\frac{1}{\delta_l} 4\biggl(\frac
{2^l}{n^{2}} +
\frac{\sum_{k \in Z_l} a_{l,k}^2 }{n}\biggr)} \biggr\} \leq \delta_l
\end{eqnarray*}
and since $\|\Pi_{W_{l}}\hat f_n\|_{2}^2 = \sum_{k \in Z_l} \hat
a_{l,k}^2$ and $\|\Pi_{W_{l}} f\|_{2}^2 = \sum_{k \in Z_l} a_{l,k}^2$ that
\begin{eqnarray*}
\Pr \biggl\{ \biggl| \|\Pi_{W_{l}} \hat f_n\|_{2}^2
- \frac{2^l}{n} - \| \Pi _{W_{l} } f\|_{2}^2 \biggr|
\geq\sqrt{\frac{1}{\delta_l} 4\biggl(\frac
{2^l}{n^{2}} +
\frac{\|\Pi_{W_{l} }f\|_{2}^2}{n}\biggr)} \biggr\} \leq \delta_l.
\end{eqnarray*}
In the same way (since there are $z_0$ terms in $Z_{J_0}$), we have for
$l = J_0$, that for any $\delta_{J_0}>0$
\begin{eqnarray*}
\Pr \biggl\{\biggl | \|\Pi_{W_{J_0}} \hat f_n\|_{2}^2
- \frac{z_0}{n} - \| \Pi_{W_{J_0} } f\|_{2}^2]\biggr |
\geq\sqrt{\frac{1}{\delta_{J_0}} 4\biggl(\frac{z_0}{n^{2}} +
\frac{\|\Pi_{W_{J_0} }f\|_{2}^2}{n}\biggr)} \biggr\} \leq \delta_{J_0}.
\end{eqnarray*}

These two last results imply by definition of $T_n(l)$, that for any
$J_0 \leq l \leq j$
\begin{eqnarray*}
\Pr \biggl\{ \bigl| T_n(l) - \|\Pi_{W_{l}} f\|_{2}^2
\bigr| \geq\sqrt {\frac
{1}{\delta_l} 4\biggl( \frac{2^l}{n^{2}} +
\frac{\|\Pi_{W_{l}}f\|_{2}^2}{n}\biggr)} \biggr\} \leq\delta_l,
\end{eqnarray*}
and
\begin{eqnarray*}
\Pr \biggl\{ \bigl| T_n(J_0) - \|\Pi_{W_{J_0}} f
\|_{2}^2 \bigr| \geq \sqrt {\frac{1}{\delta_{J_0}} 4\biggl(
\frac{z_0}{n^{2}} + \frac{\|\Pi
_{W_{J_0}}f\|
_{2}^2}{n}\biggr)} \biggr\} \leq
\delta_{J_0}.
\end{eqnarray*}
These results imply by an union bound over all $J_0 \leq l\leq j$,
that we have
\begin{eqnarray*}
&&\Pr \biggl\{\forall l\dvtx  J_0 < l \leq j, \bigl|T_n(l) - \|
\Pi_{W_l} f\| _{2}^2 \bigr| \geq\sqrt{
\frac{1}{\delta_l} 4\biggl(\frac{2^{l}}{n^{2}} + \frac
{\|\Pi_{W_{l}}f\|_{2}^2 }{n} \biggr)},
\\
&&\hphantom{\Pr \biggl\{}\bigl| T_n(J_0) - \|\Pi_{W_{J_0}} f
\|_{2}^2 \bigr| \geq\sqrt{\frac
{1}{\delta_{J_0}} 4\biggl(
\frac{z_0}{n^{2}} + \frac{\|\Pi_{W_{J_0}}f\|
_{2}^2}{n}\biggr)} \biggr\} \leq\sum
_{J_0 \leq l \leq j}\delta_l.
\end{eqnarray*}

Set for any $J_0 < l \leq j$, $\delta_l = (2^{-(j-l)/2} + 2^{-l/4})
\Delta/12$, and $\delta_{J_0} = \Delta/12$. Then
\begin{eqnarray*}
&&\Pr \biggl\{\forall l\dvtx  J_0 < l \leq j, \bigl|T_n(l) - \|
\Pi_{W_l} f\| _{2}^2 \bigr| \geq4\sqrt{
\frac{3}{\Delta} \biggl(\frac
{2^{(j+l)/2}}{n^{2}} + 2^{l/4}\frac{\|\Pi_{W_{l}}f\|_{2}^2}{n}
\biggr)},
\\
&&\hphantom{\Pr \biggl\{}\bigl| T_n(J_0) - \|\Pi_{W_{J_0}} f
\|_{2}^2 \bigr| \geq4\sqrt{\frac
{3}{\Delta} \biggl(
\frac{z_0}{n^{2}} + \frac{\|\Pi_{W_{J_0}}f\|_{2}^2}{n}\biggr)} \biggr\} \leq\sum
_{J_0 \leq l \leq j}\delta_l \leq\Delta,
\end{eqnarray*}
since
\begin{eqnarray*}
\sum_{J_0 \leq l \leq j}\delta_l \leq
\frac{\Delta}{12} + \frac
{\Delta
}{12}\sum_{1 \leq l \leq j}
\bigl(2^{-(j-l)/2} + 2^{-l/4}\bigr) \leq\frac
{\Delta
}{12} \biggl(1 +
\frac{1}{1 - 2^{-1/2}} + \frac{1}{1 - 2^{-1/4}} \biggr) \leq\Delta.
\end{eqnarray*}
Since $z_0\geq1$, we have
\begin{eqnarray*}
\Pr \biggl\{\forall l\dvtx  J_0 \leq l \leq j, \bigl|T_n(l) - \|
\Pi_{W_l} f\| _{2}^2 \bigr| \geq4\sqrt{
\frac{3z_0}{\Delta} \biggl(\frac{2^{(j+l)/2}}{n^{2}} + 2^{l/4}\frac{\|\Pi_{W_{l}}f\|_{2}^2}{n}
\biggr)} \biggr\} \leq\Delta,
\end{eqnarray*}
which concludes the proof.
\end{pf}

%s4.3 #&#
\subsection{Study of the test}

%Let $\alpha>0$ be the desired level of the test.

Set $c \equiv c(\alpha) = 24\sqrt{\frac{z_0}{\alpha}}$, where we remind
that $\alpha>0$ is the desired level of the test. By definition of the
quantities $\tau_l$ (equation \eqref{def:Tl}), we have for any $J_0<
l\leq j$
\begin{eqnarray*}
\tau_{l} \equiv\tau_{n,l} = c \frac{2^{(j+l)/8}}{\sqrt{n}}, \quad \mbox
{and}\quad  \tau_{J_0} \equiv\tau_{n,J_0} = c \frac{1}{\sqrt{n}}.
%C
%{such} \mbox{that} C \geq24\sqrt{\frac{1}{\alpha}}. %large enough
%(but depending only on $\Delta, B, s, t$).
\end{eqnarray*}
We thus have
%
%e4.3 #&#
\begin{eqnarray}
\label{eq:bloublou} \sum_{l=J_0}^j
\tau_{l} \leq\sum_{l=0}^j c
\frac
{2^{(j+l)/8}}{\sqrt {n}} \leq\frac{c}{\sqrt{n}} 2^{j/4} \biggl(1 +
\frac{1}{1 -
2^{-1/8}} \biggr) \leq14cn^{-\afrac{t}{2t+1/2}}.
\end{eqnarray}

%We recall that $j$ is such that $2^{j} = n^{1/(2t+1/2)}$ and we set
%for any $J_0 \leq l\leq j$ the positive real numbers
%{for} C \mbox{such} \mbox{that} C \geq24\sqrt{\frac{1}{\alpha}}.
%%large enough (but depending only on $\Delta, B, s, t$).
%By definition of the quantities $\tau_l$ (Equation \eqref{def:Tl}), we
%have

Also, by definition of $\tilde C(\alpha)$ in Theorem~\ref{thm:consist},
we have
\begin{eqnarray*}
\rho_n = c\biggl(\frac{2^tB}{\sqrt{1 - 2^{-2t}}}+19\biggr) n^{-t/(2t+1/2)}.
\end{eqnarray*}
In particular this implies together with equation \eqref{eq:bloublou},
and since $2^j \leq2^t n^{\afrac{t}{2t+1/2}}$, that
%
%e4.4 #&#
\begin{eqnarray}
\label{eq:verufrohn} \rho_n \geq c \frac{B}{\sqrt{1 - 2^{-2t}}}2^{-jt} + \frac{4}{3}
\sum_{J_0
\leq l
\leq j} \tau_{l}.
\end{eqnarray}

%such that $\frac{\rho_n^2}{64} + \sqrt{\frac{1}{\delta} C'(c
%, i.e. such that $\frac{D}{2}(\rho_n/4)^{p} - 2\sqrt{\frac{1}{\delta}
%C_M\frac{2^{jp( 1- 1/p)}}{n^{p}}} \geq3D_M B^{p} 2^{-pj_s s} + 2\sqrt{
%(D_p + 2 C_{p/(p-1)}\sqrt{ \log(1/\delta)}) \sqrt{\frac{ 2^{j_s}}{n}}
%2^{j_s}}{n}}$.

%We set $\tau_n^2 =\frac{\rho_n^2}{16} \Big( 1 - \sqrt{\frac{1}{\delta}
%C'} \frac{4}{\sqrt{n} \rho_n}\Big) - \sqrt{\frac{1}{\delta} C'(c

%We assume that either the basis that we consider is the Haar wavelet
%basis, or that it verifies Assumption~\ref{ass:wav}.

%s4.3.1 #&#
\subsubsection{Null hypothesis}

Since $f \in\Sigma(s,B)$, by Lemma~\ref{lem:justiftest},
\begin{eqnarray*}
\max_{J_0 \leq l\leq j} \biggl(\|\Pi_{W_l}f\|_{2} -
\frac{B}{2^{ls}} \biggr)\leq0.
\end{eqnarray*}
Thus by Lemma~\ref{lem:conc0pp}, we have with probability at least
$1-\alpha/2$ that for any $J_0 \leq l \leq j$
\begin{eqnarray*}
T_n(l) &\leq&\|\Pi_{W_{l} } f\|_{2}^2 +
4\sqrt{\frac{6z_0}{\alpha} \biggl( \frac{2^{(l+j)/2}}{n^{2}} + 2^{l/4}
\frac{\|\Pi_{W_l}f\|_{2}^2}{n}\biggr)}
\\
&\leq&\frac{B}{2^{ls}} \biggl(\frac{B}{2^{ls}} + 4 \times2^{l/4}\sqrt
{\frac{6z_0}{\alpha n}} \biggr) + 4\sqrt{\frac{6z_0}{\alpha} } \frac
{2^{(j+l)/4}}{n}
\\
&\leq&\frac{B^2}{2^{2ls}} + 4\frac{B}{2^{ls}}\sqrt{\frac
{6z_0}{\alpha
} }
\frac{2^{(j+l)/8}}{n^{1/2}}+ 4\sqrt{\frac{6z_0}{\alpha} } \frac
{2^{(j+l)/4}}{n}
\\
&\leq& \biggl(\frac{B}{2^{ls}} + 4 \sqrt{\frac{6z_0}{\alpha} } \frac
{2^{(j+l)/8}}{n^{1/2}}
\biggr)^2
\\
&\leq& \biggl(\frac{B}{2^{ls}} + \tau_{l}/\sqrt{6}
\biggr)^2 < t_n(l)^2,
\end{eqnarray*}
since $c = 24\sqrt{\frac{z_0}{\alpha}}$, and by definition of $t_n(l)$
(see equation \eqref{def:Tnl}).
%for $C$ large enough but depending only on $s,t,\Delta,B$.

So with probability at least $1-\alpha/2$, we have $\Psi_n=0$ under $H_0$.

%s4.3.2 #&#
\subsubsection{Alternative hypothesis}

The sequence $(\tau_{l})_l$, and $\rho_n$ verify the assumptions of
Lemma~\ref{lem:justiftest} (see equation \eqref{eq:verufrohn}).

If $H_1$ is verified, then
\begin{eqnarray*}
\max_{J_0 \leq l\leq j} \biggl(\bigl\|\Pi_{W_{l}}(f)\bigr\|_{2} -
\frac{B}{2^{ls}} - \tau_{l} \biggr) > 0,
\end{eqnarray*}
see Lemma~\ref{lem:justiftest}. So there exists $J_0 \leq l\leq j$
such that
\begin{eqnarray*}
\bigl\|\Pi_{W_{l}}(f)\bigr\|_{2} \geq\frac{B}{2^{ls}} +
\tau_{l}.
\end{eqnarray*}

By Lemma~\ref{lem:conc0pp}, we have with probability at least
$1-\alpha
/2$ that for this $l$
\begin{eqnarray*}
T_n(l) &\geq&\|\Pi_{W_{l}} f\|_{2}^2 -
4\sqrt{\frac{6z_0}{\alpha
} \biggl(\frac{2^{(j+l)/2}}{n^{2}} + 2^{l/4}
\frac{\|\Pi_{W_{l}}f\|_{2}^2}{
n} \biggr)}
\\
&\geq& \biggl(\frac{B}{2^{ls}} +\tau_{l} \biggr) \biggl(
\frac{B}{2^{ls}} +\tau _{l} - 4\sqrt{\frac{6z_0}{\alpha} }
\frac{2^{l/4}}{n^{1/2}} \biggr) - 4\sqrt{\frac{6z_0}{\alpha}\frac{2^{(j+l)/2}}{n^{2}}}
\\
&\geq& \biggl(\frac{B}{2^{ls}} +\tau_{l} \biggr) \biggl(
\frac{B}{2^{ls}} + \tau _{l}/2 \biggr) - 4\sqrt{\frac{6z_0}{\alpha} }
\frac{2^{(j+l)/4}}{n}
\\
&\geq&\frac{B^2}{2^{2ls}} + \frac{B}{2^{ls}}\tau_{l} +
\tau_{l}^2/2 - 4\sqrt{\frac{6z_0}{\alpha} }
\frac{2^{(j+l)/4}}{n}
\\
&\geq&\frac{B^2}{2^{2ls}} + \frac{B}{2^{ls}}\tau_{l} + \tau
_{l}^2/4
\\
&\geq& \biggl(\frac{B}{2^{ls}} + \tau_{l}/2 \biggr)^2
=t_n(l)^2.
\end{eqnarray*}
since $c = 24\sqrt{\frac{z_0}{\alpha}}$, and by definition of $t_n(l)$
(see equation \eqref{def:Tnl}).% large enough but depending only on
%$s,t,\Delta,B$.

So with probability at least $1-\alpha/2$, we have $\Psi_n=1$ under $H_1$.

%pa4.3.2.1 #&#
\noindent\textit{Conclusion on the test $\Psi_n$}.
All the inequalities developed earlier are true for any $f$ in $H_0$ or
$H_1$ with constants depending only on $s,t,B,\alpha$ and the supremum
over $f$ in $H_0$ and $H_1$ of the error of type one and two are
bounded by $\alpha/2$. Finally, the test $\Psi_n$ of errors of type 1
and 2 bounded by $\alpha/2$ distinguishes between $H_0$ and $H_1$ with
condition $\rho_n = 24\sqrt{\frac{z_0}{\alpha}}(\frac{2^tB}{\sqrt {1 -
2^{-2t}}}+19) n^{-t/(2t+1/2)}$. This implies that\vspace*{-0.5pt}
\begin{eqnarray*}
\sup_{f \in\Sigma(s,B)} \mathbb E_f \Psi_n + \sup
_{f \in\tilde
\Sigma
(t, B, \rho_n)} \mathbb E_f (1-\Psi_n) \leq
\alpha.
\end{eqnarray*}
%
%for $n$ large enough.

%By definition of $2\Delta$-consistency, this test is $2

%s5 #&#
\section{Proof of Theorem \texorpdfstring{\protect\ref{th:lb}}{3.2}}\label{sec:lowbound}

Let $B>0$, $s>t>0$, $\min(1,B)> \upsilon>0$, and $j \in\mathbb N^*$
such that $j = \lfloor1/(2t+1/2) \log(n)/\allowbreak \log(2)\rfloor$, where
$\lfloor\,\cdot\,\rfloor$ is the integer part of a real number. In
particular, this definition\vspace*{1.5pt} implies that $n^{1/(2t+1/2)}/2 \leq2^j
\leq n^{1/(2t+1/2)}$.

%pa5.subsection.subsubsection.1 #&#
\noindent\textit{Step 1}: \textit{Definition of a testing problem on some large set}.
Define the set\vspace*{-1pt}
\begin{eqnarray*}
I \equiv I_j = \bigl\{(\alpha_{l,k})_{l\geq J_0,k \in Z_l}\dvtx
\forall l \neq j, \alpha_{l,k} =0, \alpha_{j,k}\in\{-1,1\}
\bigr\}.
\end{eqnarray*}
Consider the sequence of coefficients indexed by a given $\alpha\in I$ as\vspace*{-1pt}
\begin{eqnarray*}
a_{l,k}^{(\alpha)} = \upsilon a \alpha_{l,k},
\end{eqnarray*}
where $a =\frac{1}{\sqrt{n} 2^{j/4}}$. Consider the function associated
to $a^{(\alpha)}$ that we write $f^{(\alpha)}$ and that we define as\vspace*{-1pt}
\begin{eqnarray*}
f^{(\alpha)} = \sum_{l=J_0}^{\infty} \sum
_{k \in Z_l} a^{(\alpha
)}_{l,k}
\psi_{l,k} = \sum_{k \in Z_j} a^{(\alpha)}_{j,k}
\psi_{j,k}.
\end{eqnarray*}

Consider the testing problem\vspace*{-0.5pt}
%
%e5.1 #&#
\begin{equation}
\label{test3} H_0\dvtx  f=0 \quad \mbox{vs.}\quad  H_1\dvtx  f = f^{(\alpha)},\qquad
\alpha\in I.
\end{equation}
%
%We want to prove that this test is impossible to perform if $\upsilon$
%is small.

%pa5.subsection.subsubsection.2 #&#
\noindent\textit{Step 2}: \textit{Quantity of interest}.
An observation in the white noise model is equivalent, by sufficiency
considerations, to an observation of empirical coefficients:
equivalently to having access to the process $Y^{(n)}$, we have access
to the empirical coefficients $(\hat a_{l,k})_{l,k}$ (where $\hat
a_{l,k} = \int\psi_{l,k} \,\mathrm{d}Y^{(n)}$) and each of these coefficients are
independent $\mathcal N(a_{l,k},1/n)$. Let $\Psi$ be a test, i.e., some
measurable function (according to the empirical coefficients) taking
values in $\{0,1\}$.

We have for any $\eta>0$ (using the notations $\Pr_0$ and $\mathbb E_0$
for the probability and expectation when the data are generated with $f=0$)\vspace*{-0.5pt}
%
%e5.2 #&#
\begin{eqnarray}\label{eq:boundZ}
\mathbb E_0[\Psi] + \sup_{f^{(\alpha)}, \alpha\in I}\mathbb
E_{f^{(\alpha)}}[1-\Psi] &\geq&\mathbb E_0[\Psi] +
\frac{1}{|I|} \sum_{\alpha\in I}\mathbb
E_{f^{(\alpha)}}[1-\Psi]
\nonumber
\\
&\geq&\mathbb E_0 \bigl[ \mathbf1\{\Psi=1\}\bigr] + \mathbf1 \{
\Psi=0\} Z
\\
&\geq&(1-\eta) {\Pr}_0(Z \geq1-\eta), \nonumber
\end{eqnarray}
where $Z= \frac{1}{|I|} \sum_{\alpha\in I} \prod_{l,k} \frac
{\mathrm{d}P_{l,k}^{(\alpha)}}{\mathrm{d}P_{l,k}^{0}}$, where $\mathrm{d}P_{l,k}^{(\alpha)}$ is
the density of $\hat a_{l,k}$ when the function generating the data is
$f^{(\alpha)}$, and $\mathrm{d}P_{l,k}^{0}$ is the density of $\hat a_{l,k}$
when the function generating the data is $0$ (this holds since the
$(\hat a_{l,k})_{l,k}$ are independent).

More precisely, we have since the $(\hat a_{l,k})_{l,k}$ are
independent $\mathcal N(a_{l,k},1/n)$%, and since $\alpha_k^2 = 1$, that
\begin{eqnarray*}
Z\bigl((x_{k})_{k}\bigr) &\equiv& Z\bigl((x_{l,k})_{l,k}
\bigr) = \frac{1}{|I|} \sum_{\alpha
\in I} \prod
_{l,k} \frac{\exp(\sklfrac{-n}{2}(x_{l,k} -
a_{l,k}^{(\alpha
)})^2)}{\exp(\sklfrac{-n}{2}x_{l,k}^2)}
\\
&=& \frac{1}{|I|} \sum_{\alpha\in I} \prod
_{k\in Z_j} \exp\bigl(n x_{k} a_{k}^{(\alpha)}
\bigr) \exp\biggl(- \frac{n}{2} \bigl(a_{k}^{(\alpha)}
\bigr)^2\biggr), %&= \frac{1}{|I|} \exp(\frac{1}{2} \sum_{l=j_s}^j \sum_k (\upsilon
%2^{-jt}\frac{2^{-l(1/2 - 1/p + 1/(Kp))}}{(j-j_s)^{1/p}})^2) \sum_{
%&= \frac{1}{|I|} \exp(\frac{1}{2} \sum_{l=j_s}^j \upsilon^2 2^{-2jt}
%I} \prod_{l=j_s}^j \prod_k \exp(x_{l,k} a_{l,k}^{\alpha})\\
%&= \frac{1}{|I|} \exp(\frac{1}{2} \upsilon^2 2^{-2jt}\frac{2^{-j_s(1 -
%2/p)(1 - 1/K)}}{(j-j_s)^{1/p}}) \sum_{\alpha\in I} \prod_{l=j_s}^j
%&= \frac{1}{|I|} \exp(\frac{1}{2} \big(\upsilon^2
%2^{-2jt}(j-j_s)^{1-2/p}\big)) \sum_{\alpha\in I} \prod_{l=j_s}^j
\end{eqnarray*}
where $(x_{k})_{k} \equiv(x_{j,k})_{k}$ and $(a_{k}^{(\alpha)})_k
\equiv(a_{k}^{(\alpha)})_{k}$. In the rest of the proof, we write also
$(\alpha_k)_k \equiv(\alpha_{j,k})_k$ in order to simplify notations.

By Markov and Cauchy Schwarz's inequality
%
%e5.3 #&#
\begin{eqnarray}
\label{eq:mark} {\Pr}_0(Z \geq1-\eta) \geq1 - \frac{\mathbb E_0|Z-1|}{\eta} \geq
1 - \frac{\sqrt{\mathbb E_0(Z-1)^2}}{\eta}.
\end{eqnarray}

%and note that $\upsilon^2 2^{-2jt}\frac{2^{-j_s(1 - 2/p)(1 -
%1/K)}}{(j-j_s)^{1/p}} \leq\upsilon^2 2^{-2jt}$ so these class $a^{
%2^{-2jt}\}$ which is already not distinguishable for $\upsilon$ small
%enough.

%Note also that $|a^{\alpha}|_p \geq\upsilon2^{-jt}
%impossible to distinguish for border smaller than $(\log(n))^{??} n^{-
%not the right one since we want that the coefficients are in $H_1$
%which is not the case here)?

%pa5.subsection.subsubsection.3 #&#
\noindent\textit{Step 3}: \textit{Study of the term in $Z$}.
We have by definition of $Z$
\begin{eqnarray*}
&&\mathbb E_0 \bigl[(Z - 1)^2 \bigr]
\nonumber
\\
&&\quad = \int_{x_1,\ldots, x_{2^j}} \biggl(\frac{1}{|I|} \sum
_{\alpha\in I} \prod_k \exp
\bigl(x_{k} n a_{k}^{(\alpha)} \bigr) \exp \biggl(-
\frac{n}{2} \bigl(a_{k}^{(\alpha)}\bigr)^2
\biggr) - 1 \biggr)^2 \\
&&\hphantom{\quad = \int_{x_1,\ldots, x_{2^j}}}{}\times\prod_k
\frac{1}{\sqrt{2n\uppi
}}\exp \biggl(- \frac{n}{2} (x_{k})^2
\biggr) \,\mathrm{d}x_1\cdots x_{2^j}
\nonumber
\\
&&\quad = \int_{x_1,\ldots, x_{2^j}} \biggl(\frac{1}{|I|} \sum
_{\alpha\in I} \prod_k \exp
\bigl(x_{k} n a_{k}^{(\alpha)} \bigr) \exp \biggl(-
\frac{n}{2} \bigl(a_{k}^{(\alpha)}\bigr)^2
\biggr) \biggr)^2 \\
&&\hphantom{\quad = \int_{x_1,\ldots, x_{2^j}}}{}\times\prod_k
\frac{1}{\sqrt{2n\uppi
}}\exp \biggl(- \frac{n}{2} (x_{k})^2
\biggr) \,\mathrm{d}x_1\cdots x_{2^j}
\nonumber
\\
&&\qquad {}- 2 \int_{x_1,\ldots, x_{2^j}} \frac{1}{|I|} \sum
_{\alpha\in I} \prod_k \exp
\bigl(x_{k} n a_{k}^{(\alpha)} \bigr) \exp \biggl(-
\frac{n}{2} \bigl(a_{k}^{(\alpha)}\bigr)^2
\biggr)\\
&&\hphantom{\qquad {}- 2 \int_{x_1,\ldots, x_{2^j}}}{}\times \prod_k \frac{1}{\sqrt{2n\uppi}}\exp \biggl(-
\frac
{n}{2} (x_{k})^2 \biggr) \,\mathrm{d}x_1\cdots x_{2^j}
+1
\nonumber
\\
&&\quad = \int_{x_1,\ldots, x_{2^j}} \biggl(\frac{1}{|I|} \sum
_{\alpha\in I} \prod_k \exp
\bigl(x_{k} n a_{k}^{(\alpha)} \bigr) \exp \biggl(-
\frac{n}{2} \bigl(a_{k}^{(\alpha)}\bigr)^2
\biggr) \biggr)^2 \\
&&\hphantom{\quad = \int_{x_1,\ldots, x_{2^j}}}{}\times\prod_k
\frac{1}{\sqrt{2n\uppi
}}\exp \biggl(- \frac{n}{2} (x_{k})^2
\biggr)\, \mathrm{d}x_1\cdots x_{2^j}
\nonumber
\\
&&\qquad {}- 2 \frac{1}{|I|} \sum_{\alpha\in I} \prod
_k \int_{x_k} \frac
{1}{\sqrt{2n\uppi}}\exp
\biggl(- \frac{n}{2} \bigl(x_{k} - a_k^{(\alpha
)}
\bigr)^2 \biggr) \,\mathrm{d}x_k +1
\nonumber
\\
&&\quad = \int_{x_1,\ldots, x_{2^j}} \biggl(\frac{1}{|I|} \sum
_{\alpha\in I} \prod_k \exp
\bigl(x_{k} n a_{k}^{(\alpha)} \bigr) \exp \biggl(-
\frac{n}{2} \bigl(a_{k}^{(\alpha)}\bigr)^2
\biggr) \biggr)^2 \\
&&\hphantom{\quad = \int_{x_1,\ldots, x_{2^j}}}{}\times\prod_k
\frac{1}{\sqrt{2n\uppi
}}\exp \biggl(- \frac{n}{2} (x_{k})^2
\biggr) \,\mathrm{d}x_1\cdots x_{2^j} -1
\nonumber
\end{eqnarray*}
by Fubini--Tonelli. This implies by developing the first term that\vspace*{-1pt}
\begin{eqnarray*}
&&\mathbb E_0 \bigl[(Z - 1)^2 \bigr]
\nonumber
\\[-1pt]
&&\quad = \frac{1}{|I|^2} \biggl( \sum_{\alpha, \alpha' \in I} \int
_{x_1...,x_{2^j}} \prod_k \exp \bigl(
x_{k} n \bigl(a_{k}^{(\alpha)} + a_{k}^{(\alpha')}
\bigr) \bigr) \exp \biggl(- \frac{n}{2}\bigl(\bigl(a_{k}^{(\alpha
)}
\bigr)^2 + \bigl(a_{k}^{(\alpha')}\bigr)^2
\bigr) \biggr)
\nonumber
\\[-1pt]
&&\hphantom{\quad = \frac{1}{|I|^2} \biggl( \sum_{\alpha, \alpha' \in I} \int
_{x_1...,x_{2^j}} } {}\times\frac{1}{\sqrt{2n\uppi}}\exp \biggl(- \frac{n}{2} (x_{k})^2
\biggr)\,\mathrm{d}x_1\cdots \mathrm{d}x_{2^j} \biggr) -1
\nonumber
\\[-1pt]
&&\quad = \frac{1}{|I|^2} \biggl( \sum_{\alpha, \alpha' \in I} \prod
_k \int_{x_k} \exp \bigl(
x_{k} n \bigl(a_{k}^{(\alpha)} + a_{k}^{(\alpha')}
\bigr) \bigr) \exp \biggl(- \frac{n}{2}\bigl(\bigl(a_{k}^{(\alpha)}
\bigr)^2 + \bigl(a_{k}^{(\alpha
')}\bigr)^2
\bigr) \biggr)
\nonumber
\\[-1pt]
&&\hphantom{\quad = \frac{1}{|I|^2} \biggl(\sum_{\alpha, \alpha' \in I} \prod
_k \int_{x_k}}{}\times\frac{1}{\sqrt{2n\uppi}}\exp \biggl(- \frac{n}{2} (x_{k})^2
\biggr)\,\mathrm{d}x_k \biggr) -1
\nonumber
\\[-1pt]
&&\quad = \frac{1}{|I|^2} \biggl( \sum_{\alpha, \alpha' \in I} \prod
_k \int_{x_k} \exp \bigl(
x_{k} n \upsilon a \bigl(\alpha_{k}+ \alpha_{k}'
\bigr) \bigr) \exp \bigl(- n\upsilon^2 a^2 \bigr)\\
&&\hspace*{98pt}{}\times
\frac{1}{\sqrt{2n\uppi}}\exp \biggl(- \frac
{n}{2} (x_{k})^2
\biggr)\,\mathrm{d}x_k \biggr) -1.
\nonumber
\end{eqnarray*}
This implies by integrating depending on the respective values of
$\alpha_k$ and $\alpha_k'$ that\vspace*{-1pt}
%
%e5.4 #&#
\begin{eqnarray}\label{eq:ZZZ}
&&\mathbb E_0 \bigl[(Z - 1)^2 \bigr]
\nonumber
\\[-1pt]
&&\quad = \frac{1}{|I|^2} \biggl[ \sum_{\alpha, \alpha' \in I} \prod
_k \biggl( \exp \bigl(n \upsilon^2
a^2 \bigr) \mathbf1\bigl\{\alpha_k =
\alpha_k' = 1\bigr\} \int_{x_k}
\frac{1}{\sqrt{2n\uppi}}\exp \biggl(- \frac{n}{2} (x_{k} - 2\upsilon
a)^2 \biggr)\,\mathrm{d}x_k
\nonumber
\\[-1pt]
&&\hphantom{\quad = \frac{1}{|I|^2} \biggl[ \sum_{\alpha, \alpha' \in I} \prod
_k \biggl(}{}+ \exp \bigl(n \upsilon^2 a^2 \bigr) \mathbf1\bigl\{
\alpha_k = \alpha_k' = -1\bigr\}\nonumber\\ [-1pt]
&&\hphantom{\quad = \frac{1}{|I|^2} \biggl[ \sum_{\alpha, \alpha' \in I} \prod
_k \biggl({}+}{}\times \int
_{x_k} \frac{1}{\sqrt{2n\uppi}}\exp \biggl(- \frac{n}{2}
(x_{k} + 2\upsilon a)^2 \biggr)\,\mathrm{d}x_k
\nonumber\\[-9pt]\\[-9pt]
&&\hphantom{\quad = \frac{1}{|I|^2} \biggl[ \sum_{\alpha, \alpha' \in I} \prod
_k \biggl(}{}+ \exp \bigl(-n \upsilon^2 a^2 \bigr) \mathbf1\bigl\{
\alpha_k \neq\alpha _k'\bigr\} \int
_{x_k} \frac{1}{\sqrt{2n\uppi}}\exp \biggl(- \frac{n}{2}
x_{k}^2 \biggr)\,\mathrm{d}x_k \biggr) \biggr] -1
\nonumber
\\[-1pt]
&&\quad = \frac{1}{|I|^2} \biggl( \sum_{\alpha, \alpha' \in I} \prod
_k \bigl( \exp \bigl(-n \upsilon^2
a^2 \bigr) \bigl(1 - \mathbf1\bigl\{\alpha_k \neq \alpha
_k'\bigr\} \bigr) + \exp \bigl(n \upsilon^2
a^2 \bigr) \mathbf1\bigl\{\alpha_k \neq
\alpha_k' \bigr\} \biggr) -1.\nonumber
% \nonumber\\
%&= \exp(-n 2^j\upsilon^2 a^2) \frac{1}{|I|^2} \sum_{\alpha, \alpha'
\end{eqnarray}
Since the $\alpha$ and $\alpha'$ take respectively all possible values
in $\{-1,1\}^{2^j}$, by definition of the expectation, and by replacing
$\alpha$ and $\alpha'$ by $R$ and $R'$ in the formula, we have
\begin{eqnarray*}
\frac{1}{|I|^2} \sum_{\alpha, \alpha' \in I} [\,\cdot\,] = \mathbb
E_{(R_i)_i,
(R_j')_j}[\,\cdot\,],
\end{eqnarray*}
where the $(R_i)_i, (R_j')_j$ are two sequences of i.i.d. Rademacher
random variables that are also independent of each other, and where
$\mathbb E_{(R_i)_i, (R_j')_j}[\,\cdot\,]$ is the expectation according to
these random variables. This implies together with equation \eqref
{eq:ZZZ} that
\begin{eqnarray*}
&&\mathbb E_0 \bigl[(Z - 1)^2 \bigr]
\nonumber
\\
&&\quad = \mathbb E_{(R_i)_i, (R_j')_j} \biggl[ \prod_k
\bigl( \exp \bigl(-n\upsilon^2 a^2 \bigr) \bigl(1 -
\mathbf1\bigl\{R_k \neq R_k' \bigr\}
\bigr) + \exp \bigl(n \upsilon^2 a^2 \bigr)\mathbf1\bigl
\{R_k \neq R_k' \bigr\} \bigr) \biggr] - 1
\nonumber
\\
&&\quad = \prod_k \mathbb E_{R_k, R_k'} \bigl[ \exp
\bigl(-n \upsilon^2 a^2 \bigr) \bigl(1 - \mathbf1\bigl
\{R_k \neq R_k' \bigr\} \bigr) + \exp
\bigl(n \upsilon^2 a^2 \bigr) \mathbf1\bigl
\{R_k \neq R_k' \bigr\} \bigr] - 1,
\nonumber
\end{eqnarray*}
since all $R_k$, $R_k'$ are independent of each other. Moreover,
$\mathbf1\{R_k \neq R_k' \}$ is a Bernoulli random variable of
parameter $1/2$ (since the two Rademacher are independent), which implies
\begin{eqnarray*}
\mathbb E_0 \bigl[(Z - 1)^2 \bigr] &=& \prod
_k \mathbb E_{B} \bigl[ \exp \bigl(-n
\upsilon^2 a^2 \bigr) (1-B) + \exp \bigl( n
\upsilon^2 a^2 \bigr) B \bigr] - 1
\nonumber
\\
&=& \bigl(\mathbb E_{B} \bigl[ \exp \bigl(-n \upsilon^2
a^2 \bigr) (1-B) + \exp \bigl(n \upsilon^2
a^2 \bigr) B \bigr] \bigr)^{2^j} - 1,
\nonumber
\end{eqnarray*}
where $\mathbb E_{B}[\,\cdot\,]$ is the expectation according to a Bernoulli
random variable with parameter $1/2$. The last equation implies
\begin{eqnarray*}
\mathbb E_0 \bigl[(Z - 1)^2 \bigr] &=& \biggl(
\frac{\exp (-n
\upsilon^2
a^2 ) + \exp (n \upsilon^2 a^2 )}{2} \biggr)^{2^j} - 1
\nonumber
\\
&\leq& \biggl( \frac{1 -n \upsilon^2 a^2 + (n \upsilon^2 a^2)^2 + 1 + n
\upsilon^2 a^2 + (n \upsilon^2 a^2)^2}{2} \biggr)^{2^j} - 1
\nonumber
\\
&\leq& \bigl( 1 + \bigl(n \upsilon^2 a^2
\bigr)^2 \bigr)^{2^j} - 1,
\nonumber
%&= \exp(-n2^j \upsilon^2 a^2) \Big( 1 + 2n \upsilon^2 a^2 \Big)^{2^j}
%-1. \nonumber
\end{eqnarray*}
since for any $|u| \leq1$, we have $\exp(u) \leq1 + u + u^2$. Since
$a^2 = \frac{1}{n2^{j/2}}$, we have
\begin{eqnarray*}
\mathbb E_0 \bigl[(Z - 1)^2 \bigr] &\leq& \biggl( 1 +
\frac{\upsilon^4}{2^j} \biggr)^{2^j} -1
\nonumber
\\
&\leq& \biggl(\exp\biggl(\frac{\upsilon^4}{2^j}\biggr) \biggr)^{2^j} -1 =
\exp \bigl(\upsilon ^4\bigr) - 1
\nonumber
\\
&\leq&1 + 2 \upsilon^4 - 1 = 2 \upsilon^4,
\nonumber
\end{eqnarray*}
since for any $0 \leq u \leq1$, we have $1+u \leq\exp(u) \leq1 + 2u$.

%pa5.subsection.subsubsection.4 #&#
\noindent\textit{Step 4}: \textit{Conclusion on the testing problem (\ref{test3})}.
By combining this with equations \eqref{eq:boundZ}, \eqref{eq:mark}, we
know that for $n$ large enough
\begin{eqnarray*}
\mathbb E_0[\Psi] + \sup_{f^{(\alpha)}, \alpha\in I}\mathbb
E_{f^{(\alpha)}}[1-\Psi] \geq1 - 2 \upsilon^4,
\end{eqnarray*}
and since this holds with any $\Psi$, we have
\begin{eqnarray*}
\inf_{\Psi} \Bigl[ \mathbb E_0[\Psi] + \sup
_{f^{(\alpha)}, \alpha
\in
I}\mathbb E_{f^{(\alpha)}}[1-\Psi] \Bigr] \geq1 - 2
\upsilon^4,
\end{eqnarray*}
where $\inf_{\Psi}$ is the infimum over measurable tests $\Psi$. This
implies that there is no $1 - 2 \upsilon^4$ consistent test for test
(\ref{test3}) (and it holds for any $0 \leq\upsilon<1$).

%pa5.subsection.subsubsection.5 #&#
\noindent\textit{Step 5}: \textit{Translation of this result in terms of the test
(\ref{test})}.
Set
\begin{eqnarray*}
\rho_n = \frac{\upsilon n^{-\afrac{t}{2t+1/2}}}{2}.
\end{eqnarray*}

Since $\upsilon\leq B$,
\begin{eqnarray*}
\bigl\|f^{(\alpha)}\bigr\|_{t,2, \infty} = \sqrt{\sum
_{k\in Z_l} \bigl(a_{k}^{(\alpha
)}
\bigr)^22^{2jt}} = \upsilon\leq B,
\end{eqnarray*}
so $f^{(\alpha)} \in\Sigma(t,B)$.

Also since $\forall\alpha\in I$, only the $j$th first coefficients of
$f^{(\alpha)}$ are non-zero (i.e., $f^{(\alpha)} = \Pi_{W_{j}}
(f^{(\alpha)}) = \sum_{k \in Z_j} a_{j,k}^{(\alpha)} \psi_{j,k}$), then
by definition of $\Sigma(s,B)$
\begin{eqnarray*}
\bigl\|f^{(\alpha)} - \Sigma(s,B)\bigr\|_{2} &=& \inf_{(b_{l,k})_{l,k}: 2^{ls}\|
b_{l,\cdot}\|_{l_2} \leq B}
\sqrt{\sum_{l,k} \bigl(a_{l,k}^{(\alpha)}
- b_{l,k}\bigr)^2}
\\
&=& \inf_{(b_{l,k})_{l,k}: 2^{ls}\|b_{l,\cdot}\|_{l_2} \leq B} \sqrt {\sum
_{k\in Z_j} \bigl(a_{j,k}^{(\alpha)} -
b_{j,k}\bigr)^2 + \sum_{l \neq j, k \in
Z_l}
b_{l,k}^2}
\\
&=& \inf_{(b_{k})_{k}: 2^{js}\|b\|_{l_2} \leq B} \sqrt{\sum
_{k\in Z_j} \bigl(a_{j,k}^{(\alpha)} -
b_{k}\bigr)^2 }
\\
&=& \max \biggl(0, \sqrt{\sum_{k\in Z_j}
\bigl(a_{j,k}^{(\alpha)}\bigr)^2 } - B 2^{-j
s}
\biggr)
\\
&=& \max \bigl(0, \bigl\|\Pi_{W_{j}} f^{(\alpha)}\bigr\|_{2} - B
2^{-j s} \bigr).
\end{eqnarray*}
Since by definition of the Euclidian ball, for any $u \in l_2$, we have
$\inf_{v \in l_2: \|v\|_{l_2} = 1} \|u - v\|_{l_2} = \max(0, \|u\|
_{l_2} - 1)$.

We thus have $\forall\alpha\in I$, and for all $n$ large enough
\begin{eqnarray*}
\bigl\|f^{(\alpha)} - \Sigma(s,B)\bigr\|_{2} &\geq&\bigl\|\Pi_{W_{j}}
f^{(\alpha
)}\bigr\| _{2} - B 2^{-j s}
\\
&\geq&\upsilon n^{-\afrac{t}{2t+1/2}} - B n^{-\afrac{s}{2t+1/2}}
\\
&\geq&\frac{\upsilon n^{-\afrac{t}{2t+1/2}}}{2}
\end{eqnarray*}
by triangular inequality and since for any $g \in\Sigma(s,B), \| \Pi
_{W_{j}}(g)\|_{2} \leq2^{-s}B n^{-\afrac{s}{2t+1/2}} \leq\afrac
{\upsilon
}{2} n^{\afrac{t}{2t+1/2}}$ for $n$ large enough, since $s>t$. This
together with the fact that $f^{(\alpha)} \in\Sigma(t,B)$ implies that
$\forall\alpha\in I, f^{(\alpha)}\in\tilde\Sigma(t,B,\rho_n)$.

We know that $0 \in\Sigma(s,B)$, and that $\forall\alpha,
f^{(\alpha
)}\in\tilde\Sigma(t,B,\rho_n)$ (by the previous equations). This
implies that the testing problem (\ref{test3}) is a strictly easier
problem than the testing problem (\ref{test}), that is, that
\begin{eqnarray*}
\inf_{\Psi} \Bigl[ \mathbb E_0[\Psi] + \sup
_{f^{(\alpha)}, \alpha
\in
I}\mathbb E_{f^{(\alpha)}}[1-\Psi] \Bigr] \leq\inf
_{\Psi} \Bigl[ \sup_{f
\in\Sigma(s,B)} \mathbb
E_f[\Psi] + \sup_{f \in\tilde\Sigma
(t,B,\rho
_n)}\mathbb E_{f}[1-
\Psi] \Bigr].
\end{eqnarray*}
We know that there is no $1 - 2 \upsilon^4$ consistent test for the
test (\ref{test3}) and hence, there is no $1 - 2 \upsilon^4$ consistent
test for test (\ref{test}) (and it holds for any $0 \leq\upsilon<1$).

\section*{Acknowledgements} I would like to thank Richard Nickl for
insightful discussions, as well as careful rereading and pertinent
comments. I would also like to thank Adam Bull for valuable rereading.
Finally, I would like to thank the anonymous referee for many useful
comments, as well as the Associate Editor and Editor.


%suskaldyti doi

% imsref loaded by jurgita.kaciuliene, 2014-01-31 10:40:25
% imsref loaded by jurgita.kaciuliene, 2014-01-31 10:52:36

\printhistory

\end{document}